\documentclass{amsart}
\usepackage{graphicx} 

\usepackage{a4wide}
\usepackage{color}
\usepackage{hyperref}
\usepackage[english]{babel} 
\usepackage[utf8]{inputenc} 
\usepackage[T1]{fontenc} 
\usepackage{amsmath}
\usepackage{amsfonts}
\usepackage{amssymb}
\usepackage{amsthm}
\usepackage{comment}
\usepackage{lmodern}
\usepackage{mathrsfs}
\usepackage{graphics}
\usepackage{enumerate}
\usepackage{pgf,tikz}
\usepackage{cases}
\usepackage{subfigure}
\usepackage{xfrac}

\usepackage{mathtools}

\def\R{\mathbb{R}}
\def\Z{\mathbb{Z}}
\def\C{\mathbb{C}}
\def\N{\mathbb{N}}

\def\dx{\,\mathrm{d}x}

{ 
\theoremstyle{plain}
\newtheorem{theorem}{Theorem}[section]
\newtheorem{proposition}[theorem]{Proposition}
\newtheorem{lemma}[theorem]{Lemma}
\newtheorem{corollary}[theorem]{Corollary}
\newtheorem{remark}[theorem]{Remark}

}

\title[Bifurcation of double eigenvalues for Aharonov--Bohm
operators]{Bifurcation of double eigenvalues for Aharonov--Bohm
  operators with a moving pole} \author{Laura Abatangelo \and Veronica
  Felli}

\address{Laura Abatangelo
  \newline \indent Dipartimento di Matematica
  \newline \indent Politecnico di Milano
  \newline\indent Piazza Leonardo da Vinci 32, 20133 Milano, Italy}
\email{laura.abatangelo@polimi.it}

\address{Veronica Felli
  \newline \indent Dipartimento di Matematica e Applicazioni
  \newline \indent
Universit\`a degli Studi di Milano–Bicocca
\newline\indent Via Cozzi 55, 20125 Milano, Italy}
\email{veronica.felli@unimib.it}

\date{\today}

\begin{document}

\begin{abstract}
  We study double eigenvalues of Aharonov--Bohm operators with
  Dirichlet boundary conditions in planar domains containing the
  origin. We focus on the behavior of double eigenvalues when the
  potential's circulation is a fixed half-integer number and the
  operator's pole is moving on straight lines in a neighborhood of the
  origin.  We prove that bifurcation occurs if the pole is moving
  along straight lines in a certain number of cones with positive
  measure. More precise information is given for symmetric
    domains; in particular, in the special case of the disk, any
  eigenvalue is double if the pole is located at the centre, but there
  exists a whole neighborhood where it bifurcates into two distinct
  branches.
\end{abstract}

\maketitle

\section{Introduction}

For any $a=(a_1,a_2) \in \R^2$, the Aharonov--Bohm potential with pole
$a$ and circulation $1/2$ is defined as
\[
A_a(x_1,x_2):=\frac12 \left( - \frac{x_2 - a_2}{(x_1 - a_1)^2 + (x_2 - a_2)^2}, 
\frac{x_1 - a_1}{(x_1 - a_1)^2 + (x_2 - a_2)^2}\right).
\]
The vector potential $A_a$ produces a $\delta$-type magnetic field,
which gives rise to the so-called \emph{Aharonov--Bohm effect} in
quantum mechanics, see \cite{adami-teta,aharonov-bohm}: a particle is
affected by the presence of a magnetic field, even if this is zero
almost everywhere. This excited a great interest in the physicists'
community, as the vector potential attached to the magnetic field was
likely responsible for the particle's dynamics.

The number $1/2$ in front of the vector field is the circulation of
$A_a$, up to a normalization of $2\pi$: in this regard, we mention
that the spectrum of the operator would be the same for any
circulation in $\tfrac12+{\Z}$, due to gauge equivalence of the
corresponding vector potentials, see \cite[Theorem
1.2]{L_gauge_invariance} and \cite[Proposition 2.2]{Lena2015}.

We are interested in the 
spectral properties of the magnetic Laplacian with Aharonov--Bohm 
vector potential
\begin{equation} \label{eq:operator}
(i\nabla + A_a)^2 u := -\Delta u + 2 i A_a \cdot 
\nabla u + |A_a|^2 u,
\end{equation}
acting on functions $u \, : \, \R^2 \to \C$.  From the point of view
of Functional Analysis, the differential operator \eqref{eq:operator}
can not be considered as a lower order perturbation of the standard
Laplacian, since the potentials involved are out of the Kato class.
Let us remark that, in principle, one can consider any circulation
$\kappa \in (0,1)$ for the vector potential, but the case $\kappa=1/2$
enjoys very special features which are fundamental in our
analysis. One above all, magnetic problems involving operator
\eqref{eq:operator} correspond to real Laplace problems on the double
covering manifold. Therefore operator \eqref{eq:operator} with
half-integer circulation behaves like a real operator, although it
acts on complex-valued functions, see \cite{HHOO99}. This means,
for example, that the nodal set of eigenfunctions of
\eqref{eq:operator} are made of curves instead of being isolated
points, as in general they would be in case of complex-valued
eigenfunctions, see \cite{FFT}.  We also mention that the
  case of half-integer circulations has a particular interest from the
  mathematical point of view due to applications to the problem of
  spectral minimal partitions, see \cite{BNHHO2009, NT}. Some papers
studying the behavior of simple eigenvalues in the case of
non-half-integer circulations appeared recently: \cite{AFNN2018} deals
with a single pole which is moving along straight lines, and
\cite{FNS2024} considers the case of multiple colliding poles with any
circulation. Other general results on spectral stability and
  regularity of the eigenvalue variation are contained in the paper
\cite{Lena2015}.

Let $\Omega\subset\R^2$ be a bounded, open and  connected
domain.
For any $a\in\Omega$,  we consider the eigenvalue
problem 
\begin{equation}\label{eq:eige_equation_a}\tag{$E_a$}
  \begin{cases}
   (i\nabla + A_{a})^2 u = \lambda u,  &\text{in }\Omega,\\
   u = 0, &\text{on }\partial \Omega,
 \end{cases}
\end{equation}
in the weak sense which is specified in the following lines. 
The functional space
$H^{1 ,a}(\Omega,\C)$ is the completion of
\[
\{u\in
H^1(\Omega,\C)\cap C^\infty(\Omega,\C):u\text{ vanishes in a
  neighborhood of }a\}
\]
 with respect to the norm 
 $$
 \|u\|_{H^{1,a}(\Omega,\C)}=\left(\left\|\nabla u\right\|^2
   _{L^2(\Omega,\C^2)} +\|u\|^2_{L^2(\Omega,\C)}+\big\|\tfrac{u}{|x-a|}
   \big\|^2_{L^2(\Omega,\C)}\right)^{\!\!1/2},
$$
which, in view of the
Hardy type inequality proved in \cite{LW99} (see also \cite[Lemma 3.1
and Remark 3.2]{FFT}), is 
equivalent to the norm 
\begin{equation*}
  \left(\left\|(i\nabla+A_{a}) u\right\|^2
    _{L^2(\Omega,\C^2)} +\|u\|^2_{L^2(\Omega,\C)}\right)^{\!\!1/2}.
\end{equation*}
We denote as $H^{1 ,a}_{0}(\Omega,\C)$ the space obtained as the
closure of $C^\infty_{\rm c}(\Omega\setminus\{a\},\C)$ in
$H^{1,a}(\Omega,\C)$.

For every $a\in\Omega$,  we say that $\lambda\in\R$ is an eigenvalue
of problem \eqref{eq:eige_equation_a}
in a weak sense if there exists $u\in
H^{1,a}_{0}(\Omega,\C)\setminus\{0\}$ (called an eigenfunction) such that
\begin{equation}\label{eq:weaksense}
\int_\Omega (i\nabla u+A_{a} u)\cdot \overline{(i\nabla v+A_{a}
  v)}\,dx=\lambda\int_\Omega u\overline{ v}\,dx \quad\text{for all }v\in H^{1,a}_{0}(\Omega,\C).
\end{equation}
From classical spectral theory (see e.g. \cite[Chapter 6]{B}), the
eigenvalue problem $(E_a)$ admits a sequence of real diverging
eigenvalues
$\lambda_1^a \leq \lambda_2^a\leq\ldots\leq \lambda_j^a\leq\ldots$
(repeated according to their finite multiplicity).  We will study the
spectrum, depending on the position of the pole $a$, in a neighborhood
of a fixed point $b\in \Omega$: without loss of generality, we can
consider $b=0\in\Omega$.

In particular, we are interested in genericity properties of simple
eigenvalues with respect to the location of the pole $a$ near $0$,
where we assume that a certain eigenvalue is double.  A preliminary
result on this topic is provided by the paper
\cite{AbatangeloNys2018}, where sufficient conditions are detected in
order to have double eigenvalues locally only at $0$. That said,
the aforementioned result has several drawbacks: it can be applied to
a limited number of cases, for the sufficient conditions are quite
restrictive and can be hardly checked. In this paper, we aim to
  advance the analysis, by finding a large (not negligible)
  portion of a neighborhood of the origin where a given double
  eigenvalue bifurcates into two simple branches, without the
  additional restrictive assumptions required in
  \cite{AbatangeloNys2018}.

\bigskip

Before stating our main results, we recall some basic known facts
about stability of eigenvalues under small displacements of the pole.
\begin{theorem}{\rm(\cite[Theorem 1.1, Theorem 1.3]{BonnaillieNorisNysTerracini2014},
\cite[Theorem 1.2, Theorem 1.3]{Lena2015})} \label{thm:reg-a}
Let $\Omega \subset \R^2$ be open, bounded 
and connected. Fix any $j \in \N \setminus\{0\}$. The map 
$a \in \Omega \mapsto \lambda_j^{a}$ has a continuous 
extension on $\overline{\Omega}$, that is 
\[
\lambda_j^{a} \to \lambda_j^{b} 
\quad \text{ as } a \to b \in \Omega 
\qquad \text{ and } \qquad
\lambda_j^{a} \to \lambda_j 
\quad \text{ as } \mathop{\rm dist}(a,\partial\Omega)\to 0,
\]
where $\lambda_j$ is the $j$-th eigenvalue of the Laplacian in $\Omega$ with 
Dirichlet boundary conditions. 

Moreover, if $b \in \Omega$ and the eigenvalue $\lambda_j^{b}$ is
simple, the map $a \in \Omega \mapsto \lambda_j^{a}$ is analytic in a
neighborhood of $b$.
\end{theorem}

\begin{remark}\label{rem:analiticita-multipli}
  The Kato-Rellich perturbation theory gives some information 
  on higher regularity of the eigenvalue map
  even
  when the limit eigenvalue is not simple.  Let $m$ be the
  multiplicity of the eigenvalue $\lambda^0_j$ and
  $a=t(\cos\alpha,\sin\alpha)\to0$ as $t\to0$ along a fixed direction
  $(\cos\alpha,\sin\alpha)$. In this case, there exist $m$ (not
  necessarily distinct) functions
  $t\mapsto \tilde\lambda_1(t), \ldots, t\mapsto \tilde\lambda_m(t)$
  such that $\tilde\lambda_j(0)=\lambda^0_j$ and $\tilde\lambda_j$ is analytic in a neighborhood of $0$,
   for every $j=1,\dots,m$. Moreover, 
  for any small $\varepsilon>0$, the
  eigenvalues of $(i\nabla + A_{t(\cos\alpha,\sin\alpha)})^2$ lying in
  the neighbourhood $(\lambda^0_j-\varepsilon, \lambda^0_j+\varepsilon)$ are precisely
  $\{\tilde\lambda_j(t): j=1,\dots,m\}$, provided $t$ is sufficiently
  small (see \cite[Section IV]{Lena2015} and \cite[VII, \S 3.5, Theorem 3.9]{kato}).
\end{remark}

In order to study the multiplicity of eigenvalues for poles located in
a  neighborhood  of the  origin,  at  which  a certain  eigenvalue  is
supposed to be double, we analyze the eigenbranches as the pole
moves along a straight line.  In particular, we wonder under what
  circumstances the  eigenbranches could be separated  from each
  other.  To this aim, we study the asymptotic
  expansion of the eigenvalue variation as the pole moves.

  We recall that, for simple eigenvalues, the eigenbranch's expansion
  strongly relies on the asymptotic behavior of the corresponding
  $L^2$-normalized eigenfunction, as the papers
  \cite{AbatangeloFelli2015, AbatangeloFelli2016} show. This also
    has a significant impact on the study of multiple eigenvalues
    developed in \cite{Abatangelo2019,AbatangeloNys2018}.

    In what follows, we review what is known about asymptotic behavior
    of eigenfunctions near the pole. By \cite[Theorem 1.3]{FFT} (see
    also \cite[Proposition 2.1]{AbatangeloFelli2015}), if
    $\varphi\in H^{1,0}_{0}(\Omega,\C)\setminus\{0\}$ is a an
    eigenfunction of problem $(E_0)$, then
\begin{equation*}
  \varphi \text{ has at $0$ a zero
    of order $\frac k2$ for some odd $k\in \N$},
\end{equation*}
and there exist $\beta_1,\beta_2\in\C$ such that
$(\beta_1,\beta_2)\neq(0,0)$ and
\begin{equation}\label{eq:131-old}
  r^{-k/2} \varphi(r\cos t,r\sin t) \to 
  e^{i\frac t2}\left(\beta_1
    \cos\big(\tfrac k2
    t\big)+\beta_2 
    \sin\big(\tfrac k2
    t\big)\right) \quad \text{as $r\to0^+$ in }C^{1,\tau}([0,2\pi],\C),
\end{equation}
for all $\tau\in(0,1)$.  Furthermore, in view of \cite[Lemma
3.3]{HHOO99}, in every eigenspace of $(E_0)$, we can choose a basis
consisting of eigenfunctions $\varphi$ enjoying the property
\begin{equation}\label{eq:propertyP}\tag{$K$}
 e^{-i\frac t2}\varphi(r\cos t, r\sin t)\text{ is a real-valued function};
\end{equation}
see also \cite{HHOHOO2000} and \cite[\S 10.8.2]{BNH2017}, where
functions with the property \eqref{eq:propertyP} are referred to as
$K$-real.  If an eigenfunction
$\varphi\in H^{1,0}_{0}(\Omega,\C)\setminus\{0\}$ of $(E_0)$ satisfies
\eqref{eq:propertyP}, then the asymptotics \eqref{eq:131-old} can be
rewritten as follows: there exist $\beta\in \R\setminus\{0\}$ and
$\omega\in \big[0,\frac{2\pi}{k}\big)$ such that
\begin{equation}\label{eq:131}
  r^{-k/2} \varphi(r\cos t,r\sin t) \to 
  \beta e^{i\frac t2}
  \sin\big(\tfrac k2(t-\omega)\big) \quad \text{in }C^{1,\tau}([0,2\pi],\C)
\end{equation}
as $r\to0^+$ for all $\tau\in (0,1)$, see Proposition \ref{p:asyeige}.
Therefore, $\varphi$ has exactly $k$ nodal lines meeting
 at $0$ and dividing the whole angle into $k$ equal parts.

\bigskip

We are now in position to outline the results of the present paper.
The main one establishes a ramification from a double eigenvalue, if
the pole is moving on some special straight lines.  More precisely,
there exists at least a cone of positive measure such that, as the
pole moves along a straight line within the cone (and its
  rotated with periodicity $\frac\pi k$), a bifurcation occurs and
the double eigenvalue splits into two simple eigenvalues in a
neighborhood of the origin, see Figure \ref{f:rami}.
\begin{theorem}\label{t:genericity}
  Let $N\in\N\setminus\{0\}$ be such that
  $\lambda_{N-1}^0< \lambda_{N}^0=\lambda_{N+1}^0<\lambda_{N+2}^0$, so
  that $\lambda_{N}^0$ is a double eigenvalue of $(E_0)$. Then there
  exist an odd natural number $k$ and an interval $I\subset \R$ with
  positive length $|I|>0$ such that, for every
  $\alpha \in I+\frac{\pi}k \Z$, 
    \begin{equation*}
    \lambda_{N}^a<\lambda_{N+1}^a
    \quad\text{for all $a=|a|(\cos\alpha,\sin\alpha)$ with $|a|$ sufficiently small}.
    \end{equation*}
    More precisely, one of the following three
  alternative situations 
  occurs:
\begin{itemize}
\item[(i)]
      $\lambda_{N}^a=\lambda_N^0+\mu_1(\alpha)|a|^k+o(|a|^k)$,
      $\lambda_{N+1}^a=\lambda_N^0+\mu_2(\alpha)|a|^k+o(|a|^k)$ 
       as $a=|a|(\cos\alpha,\sin\alpha)\to0$, for every
  $\alpha \in I+\frac{\pi}k \Z$ and
      some $\mu_1(\alpha)<\mu_2(\alpha)$;
  \item[(ii)]
      $\lambda_{N}^a=\lambda_N^0+\mathcal C(\alpha)|a|^k+o(|a|^k)$ and
      $\lambda_{N+1}^a=\lambda_N^0+o(|a|^k)$ 
       as $a=|a|(\cos\alpha,\sin\alpha)\to0$, for every
  $\alpha \in I+\frac{\pi}k \Z$ and
      some $\mathcal C(\alpha)<0$ ;
  \item[(iii)]
      $\lambda_{N}^a=\lambda_N^0+o(|a|^k)$ and
      $\lambda_{N+1}^a=\lambda_N^0+\mathcal C(\alpha)|a|^k+o(|a|^k)$ 
       as $a=|a|(\cos\alpha,\sin\alpha)\to0$, for every
  $\alpha \in I+\frac{\pi}k \Z$ and
      some $\mathcal C(\alpha)>0$.
    \end{itemize} 
  \end{theorem}
The number $k$ appearing in the
statement of Theorem \ref{t:genericity} is in fact twice the lowest
vanishing order of eigenfunctions relative to the double eigenvalue
$\lambda_0^N$. We observe that $k$ does not depend on the choice of
the eigenbasis.
This will be clear
throughout the proof.

We note
  that in the previous theorem the maximal  size of $|a|$ to have
  bifurcation depends on
  $\alpha$ fixed. That being the case, we can not deduce that the
  ramification occurs \emph{uniformly} in the cone, as the infimum of the
  admissible values of $|a|$ may be zero. Whether bifurcation is
  uniform with respect to
  $\alpha$ is still an open problem. 

\begin{figure}[ht]
\centering
\begin{tikzpicture}[scale=0.5]
\draw [fill=black] (0,0) circle (3pt);
\draw [fill=black] (3,0) circle (3pt);
\filldraw[color=gray!40, fill=gray!20]  plot [smooth cycle]
coordinates {(9,0) (8,2) (4,2.5) (0,1) (-2,3) (-5,3) (-6,2) (-5,0) (-6,-3) (0,-2) (2,-2)};
  \draw [->,line width=1pt](3,0)--(2.5,0);
\fill[line width=1pt,color=gray,fill=gray!40] (0.,0.) --
(-4.5,-0.525) to [out=180, in=180] (-3.45,0.555) -- cycle;
\fill[line width=1pt,color=gray,fill=gray!40] (0.,0.) --
(6,-0.7) to [out=0, in=0] (6.57,0.74) -- cycle;
\fill[line width=1pt,color=gray,fill=gray!40] (0.,0.) --
(-4.5,2.3) to [out=150, in=150] (-5,1.5) -- cycle;
\fill[line width=1pt,color=gray,fill=gray!40] (0.,0.) --
(3.15,-1.61) to [out=0, in=0] (3.5,-1.05) -- cycle;
\fill[line width=1pt,color=gray,fill=gray!40] (0.,0.) --
(-4.5,-2.3) to [out=190, in=190] (-5,-1.5) -- cycle;
\fill[line width=1pt,color=gray,fill=gray!40] (0.,0.) --
(4.5,2.3) to [out=20, in=20] (5,1.5) -- cycle;
\draw[domain=0:2.5, smooth, variable=\x, blue,line width=1pt] plot ({\x}, {0.3*\x*\x});
\draw[domain=-2.5:0, smooth, variable=\x, red,line width=1pt] plot ({\x}, {-0.3*\x*\x});
\draw[domain=0:2.5, smooth, variable=\x, red, line width=1pt] plot ({\x}, {0.6*\x*\x});
\draw[domain=-2.5:0, smooth, variable=\x, blue, line width=1pt] plot ({\x}, {-0.6*\x*\x});
\draw[color=black] (5.7,0) node {\scriptsize $a=|a|(\cos\alpha,\sin\alpha)$};
\draw[color=black] (3.7,1.8) node {\color{blue}\scriptsize $\lambda_{N}^a-\lambda_{N}^0$};
\draw[color=black] (4.1,3.9) node {\color{red} \scriptsize $\lambda_{N+1}^a-\lambda_{N}^0$};
\draw[color=black] (0,-0.4) node {\scriptsize $0$};
\end{tikzpicture}
\caption{Ramification of double eigenvalues.}
\label{f:rami}
\end{figure}
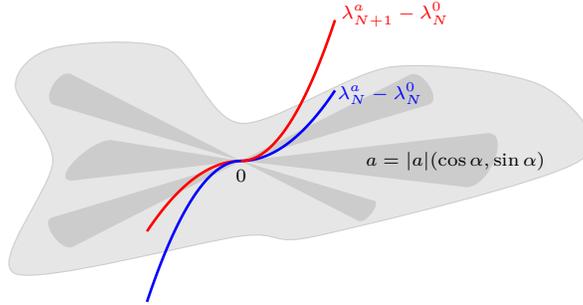

On the other hand, we conjecture that, in fact, the bifurcation occurs
for the pole moving along almost every direction. The validity of this
stronger result will be the object of a future investigation, as it
appears to require some local concavity/convexity properties for the
function $\mathcal C$ defined in \eqref{d:functionC}, which do not
seem easy to prove.

Anyway, Theorem \ref{t:genericity} can be improved in case of
symmetric domains, for example when $\Omega$ is axially or
rotationally symmetric.  In particular, as a first consequence of
Theorem \ref{t:genericity}, we have the following corollary in the
case of the domain $\Omega$ being invariant under a rotation of a
fixed angle $\tfrac{2\pi}{\ell}$, with $\ell\in \N\setminus\{0\}$:
under such a symmetry condition, bifurcation occurs along a larger set
of directions.
\begin{corollary}\label{c:rotdom}
  Let $N\in\N\setminus\{0\}$ be such that
  $\lambda_{N-1}^0< \lambda_{N}^0=\lambda_{N+1}^0<\lambda_{N+2}^0$, so
  that $\lambda_{N}^0$ is a double eigenvalue of $(E_0)$.  Let
  $\ell\in \N\setminus\{0\}$ be such that $\Omega$ is invariant under
  a rotation of an angle $\tfrac{2\pi}{\ell}$. Then there exist an odd
  natural number $k$ and an interval $I\subset \R$ with positive
  length $|I|>0$ such that, for every
  $\alpha \in I+(\frac{\pi}k +\frac{2\pi}\ell) \Z$,
    \begin{equation*}
    \lambda_{N}^a<\lambda_{N+1}^a
    \quad\text{for all $a=|a|(\cos\alpha,\sin\alpha)$ with $|a|$ sufficiently small}.
    \end{equation*}
\end{corollary}

The following result concerns axially symmetric domains and guarantees
that bifurcation into two simple eigenvalues occurs if the pole moves
along the symmetry axis.

\begin{corollary}\label{c:symmdom}
  Let $N\in\N\setminus\{0\}$ be such that
  $\lambda_{N-1}^0< \lambda_{N}^0=\lambda_{N+1}^0<\lambda_{N+2}^0$, so
  that $\lambda_{N}^0$ is a double eigenvalue of $(E_0)$.  Let
  $\Omega$ be symmetric with respect to the $x_1$-axis,
  i.e. $(x_1,x_2)\in\Omega$ if and only if $(x_1,-x_2)\in\Omega$.
  Then the interval $I$ provided by Theorem \ref{t:genericity}
  contains $0$.
  \end{corollary}
In particular, the above corollary points out that the union of
cones where ramification occurs intersects the symmetry axis
$x_2=0$, as the angles $\alpha=0,\pi$ are included.

\begin{figure}[ht]
\centering
\begin{tikzpicture}[scale=0.5]
  \filldraw[color=gray!40, fill=gray!20]  plot [smooth cycle] coordinates
  {(4,0) (5,0.5) (7,1) (8,2) (5,2) (3,1) (1,1) (1,2) (-2,2) (-3,1) (-3,0.5) (-4,0)
(-5,-0.5) (-7,-1) (-8,-2) (-5,-2) (-3,-1) (-1,-1) (-1,-2) (2,-2) (3,-1) (3,-0.5)
};
\draw[line width=0.1pt,->](-8,1) -- (8,-1);
\draw[line width=0.1pt,->](-4,-2) -- (4,2);
\draw[color=black] (8,-0.7) node {\tiny $x_1$};
\draw[color=black] (3.6,2.2) node {\tiny $x_2$};
\draw[domain=-4:4, smooth, variable=\x, red,line width=1pt] plot ({\x}, {0.3*\x*\x});
\draw[domain=-4:4, smooth, variable=\x, blue,line width=1pt] plot ({\x}, {-0.3*\x*\x});
\draw[color=black] (5,-3.9) node {\color{blue}\scriptsize $\lambda_{N}^a-\lambda_{N}^0$};
\draw[color=black] (5.2,3.9) node {\color{red} \scriptsize $\lambda_{N+1}^a-\lambda_{N}^0$};
\draw [->,line width=1pt](2.4,-0.3)--(1.4,-0.175);
\draw [fill=black] (2.4,-0.3) circle (3pt);
\draw[color=black] (2.4,-0.7) node {\scriptsize $a$};
\end{tikzpicture}
\caption{Domains with the symmetries of a
rectangle.}
\label{f:rectangle}
\end{figure}
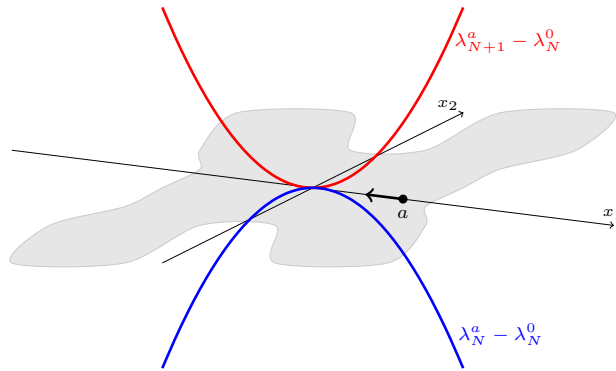

An interesting example is given by domains with the symmetries of a
rectangle.  For such domains we observe that a double eigenvalue of
the problem with a pole located in the center of symmetry splits into
two different eigenbranches, as the pole moves along one of the two
symmetry axes: one is below and the other is above the limit
eigenvalue, see Figure \ref{f:rectangle}.
\begin{corollary}\label{cor:simmetria-rettangolo}
  Let
  $\Omega$
   have the symmetries of a rectangle,
  i.e. $(x_1,x_2)\in\Omega$ if and only if $(x_1,-x_2)\in\Omega$ and $(-x_1,x_2)\in\Omega$.
  Let $N\in\N\setminus\{0\}$ be such that
  $\lambda_{N-1}^0< \lambda_{N}^0=\lambda_{N+1}^0<\lambda_{N+2}^0$, so
  that $\lambda_{N}^0$ is a double eigenvalue of $(E_0)$.  Then
  there exist an odd natural number $k$ and an open interval
    $I\subset \R$ with positive length $|I|>0$ such that $0\in I$ and,
    for every $\alpha \in I+\frac{\pi}2 \Z$,
    \begin{equation*}
      \lambda_{N}^a=\lambda_N^0-\mu(\alpha)|a|^k+o(|a|^k)\quad
      \text{and}\quad \lambda_{N+1}^a=\lambda_N^0+\mu(\alpha)|a|^k+o(|a|^k),
    \end{equation*}
    as $a=|a|(\cos\alpha,\sin\alpha)\to0$, for some positive constant
    $\mu(\alpha)>0$ depending on $\alpha$ such that $\mu(\alpha+\pi)=\mu(\alpha)$.
  \end{corollary}

The most favorable situation from the point of view of symmetries
  is obviously represented by the disk
  $D:= \{(x_1,x_2) \in \R^2:=x_1^2+x_2^2=1 \}$, which is invariant
  under rotation by any angle and reflection through any of its
  diameters. It is known from \cite[Lemma 2.4]{Abatangelo2019}, see
  also \cite[Proposition 5.3]{BNHHO2009}, that all eigenvalues
  $\lambda_j^a$ of problem \eqref{eq:eige_equation_a} in $\Omega=D$
  are double if the pole $a$ is located at the origin. On the other
  hand, for the
  first eigenvalue it is proved  in \cite{Abatangelo2019} that
  $\lambda_1^a$ is simple if $a\in D\setminus\{0\}$. The analysis
  carried out in the present paper shows that the
  subsequent eigenvalues $\lambda_j^a$ with $j\geq2$ are also simple
  locally around $0$, but not at the origin.  

\begin{corollary}\label{cor:disk}
  Let $\Omega=D$ be the unit disk and $\lambda$ be any
  eigenvalue of problem $(E_0)$. Then $\lambda$ is double, i.e.
  $\lambda_{N-1}^0<\lambda=\lambda_N^0=\lambda_{N+1}^0<\lambda_{N+2}^0$
  for some $N\in \N\setminus\{0\}$. Moreover,
  there exist a neighborhood $\mathcal U\subset \R^2$ of the
  origin such that $\lambda_{N+i}^a$ is simple for any $i=0,1$ and
  $a\in \mathcal U\setminus \{0\}$. More precisely, there exist
  $k\in\N$ odd and a positive constant $\mathcal M>0$ 
  such that the
  two different eigenbranches departing from $\lambda_N^0$ have the
  following asymptotic expansion, as $a$ is moving along any radius and
  $|a|\to0$:
    \begin{align*}
&\lambda_{N}^a=\lambda_N^0-\mathcal M |a|^{k}+o(|a|^{k}),\\
&\lambda_{N+1}^a=\lambda_N^0 + \mathcal M |a|^{k}+o(|a|^{k}).
    \end{align*}
\end{corollary}

We mention that a lot of numerical simulations are available in the
literature showing the behavior of Aharonov--Bohm eigenvalues in
symmetric domains, with poles moving along axes of symmetry.  We refer
the reader to \cite{BonnaillieHelffer2011,
  BonnaillieHelffer2013, BNH2017, BonnaillieNorisNysTerracini2014}. It is
worthwhile to notice that our results validate particularly the
simulations on the disk \cite{BonnaillieHelffer2013}, where any
eigenvalue is seen to be simple if the pole is locally away from the
origin.

\bigskip

The paper is organized as follows. In Section \ref{sec:preliminaries}
we set the variational framework of the problem. In particular, we
specify how to work with \emph{real} problems, even if in principle
eigenfunctions are complex-valued; indeed, a gauge transformation
makes \eqref{eq:eige_equation_a} equivalent to an eigenvalue problem
for the Laplacian in a domain with a straight crack along the moving
direction of the pole, in the spirit of
\cite{FelliNorisOgnibeneSiclari2023}.  We also recall from
\cite{FelliNorisOgnibeneSiclari2023} the definition of some key
quantities, obtained as minima of some energy functionals and
appearing in the asymptotic expansions of Aharonov-Bohm eigenvalues.
Section \ref{sec:appLemma} is devoted to the general approach of the
perturbed problem; in particular, Theorem \ref{thm:approxEVs} provides
a preliminary eigenvalue expansion in terms of the eigenvalues of the
bilinear form \eqref{eq:def-r_a}, defined on the finite-dimensional
eigenspace associated with the (possibly multiple) limit
eigenvalue. In Section \ref{sec:functionC} we study the key function
\eqref{d:functionC} of the angle $\alpha$ appearing in the asymptotic
expansion of eigenbranches, as the pole
$a=|a|(\cos\alpha, \sin\alpha)$ moves towards the origin along a fixed
half-line with slope $\alpha$.  Introducing a suitable decomposition
of the limit eigenspace in Section \ref{sec:decomposition}, in Section
\ref{sec:ramification} we observe a bifurcation phenomenon from a
double eigenvalue along certain directions, thus proving our main
result Theorem \ref{t:genericity}. Finally, Section
\ref{sec:symmetric} deals with some relevant applications in symmetric
domains, validating several preexisting numerical simulations.

\section{Preliminaries}\label{sec:preliminaries}

\subsection{An equivalent eigenvalue problem by gauge transformation}
A suitable gauge transformation allows us to obtain an equivalent
formulation of \eqref{eq:eige_equation_a} as an eigenvalue problem for
the Laplacian in a domain with straight cracks. We are now going to
introduce the notation and context as in
\cite{FelliNorisOgnibeneSiclari2023}.

For any  $\alpha\in (-\pi,\pi]$,  we let  $\nu_\alpha:=(-\sin\alpha,\cos\alpha)$
and  consider the half-planes
\[
\pi_\alpha^+:=\{ x\in\R^2:\ x\cdot \nu_\alpha >0 \} 
\quad \text{and}\quad 
\pi_\alpha^-:=\{ x\in\R^2:\ x\cdot \nu_\alpha <0 \}. 
\]
Then $\nu_\alpha$ is the unit outer normal vector to $\pi_\alpha^-$ on $\partial\pi_\alpha^-$. 
We also define 
\begin{equation*}
     \Sigma_\alpha := \partial\pi_\alpha^+=\partial\pi_\alpha^-=\{ t(\cos\alpha,\sin\alpha): t\in \R \}.
\end{equation*}
In view of classical trace results and embedding theorems for
fractional Sobolev spaces in dimension~1, for every
$\alpha\in(-\pi,\pi]$ and $p \in [2, +\infty)$ there exist continuous
trace operators
\[
\gamma^\alpha_+:  H^1(\pi^+_\alpha) \to L^p(\Sigma_\alpha)
\quad \text{and}\quad
\gamma^\alpha_-: H^1(\pi^-_\alpha) \to L^p(\Sigma_\alpha).
\]
For a function $u\in H^1(\Omega\setminus \Sigma_\alpha)$, we will simply write $\gamma^\alpha_+(u)$
and $\gamma^\alpha_-(u)$ to indicate 
$\gamma^\alpha_+(u\big|_{\pi_\alpha^+})$ and $\gamma^\alpha_-(u\big|_{\pi_\alpha^-})$, respectively.

For every $a\in\R^2\setminus\{0\}$ we define
\begin{equation*}
    \Gamma_a :=\{ ta: t\in (-\infty,1] \}\quad\text{and}\quad
    S_a :=\{ t a: t\in [0,1] \},
\end{equation*}
so that, if $a\in\R^2\setminus\{0\}$ belongs to the half-line from $0$
with slope $\alpha$, i.e.  $a=|a|(\cos \alpha,\sin\alpha)$ with
$|a|>0$, $\Gamma_a :=\{ t(\cos\alpha,\sin\alpha): t\leq |a|\}$ and
$S_a :=\{ t(\cos\alpha,\sin\alpha): 0\leq t\leq |a|\}$.

 We also consider the functional space $\mathcal{H}_a$ defined as the
closure of 
\begin{equation*}
    \{ w\in H^1(\Omega\setminus \Gamma_a): \ w=0 \text{ in a neighborhood of }\partial\Omega \}
\end{equation*} 
in $H^1(\Omega\setminus \Gamma_a)$.
The norm 
\[
\|w\|_{\mathcal H_a} = \left( \int_{\Omega\setminus \Gamma_a} |\nabla w|^2\,dx \right)^{1/2}
\]
on $\mathcal{H}_a$ is equivalent to the standard
$H^1(\Omega\setminus\Gamma_a)$-norm by the Poincar\'e inequality.  Let
$\widetilde{\mathcal H}_a$ denote the subspace
\[
  \widetilde{\mathcal H}_a := \{ w\in \mathcal{H}_a : \
  \gamma^\alpha_+(w) + \gamma^\alpha_-(w)=0 \text{ on }\Gamma_a \},
\]
being $\alpha$ the slope of $\Gamma_a$.

As detailed in \cite[Section 3]{FelliNorisOgnibeneSiclari2023}, 
if $a=|a|(\cos\alpha,\sin\alpha)\in \R^2\setminus\{0\}$, the function 
\begin{align*}
&\Theta_a: \R^2\setminus\{a\}\to \R,\\
&\Theta_a(a+(r\cos t,r\sin t))=\frac t2\quad\text{for every }r>0\text{ and }t\in [\alpha-\pi,\alpha+\pi),
\end{align*}
i.e.
\begin{equation*}
    \Theta_a(a+(r\cos t,r\sin t))=
    \begin{cases}
       \frac 12 t ,&\text{if }t\in[0,\pi+\alpha),\\[3pt]
       \frac12 t-\pi,&\text{if }t\in[\pi+\alpha,2\pi),
    \end{cases}
\end{equation*}
enjoys the following properties:
\begin{align*}
  &\Theta_a\in C^\infty(\R^2\setminus\Gamma_a),\\
  &\nabla \Theta_a \text{ can be continued to be in
    }C^\infty(\R^2\setminus\{a\})
    \text{ with }\nabla\Theta_a=A_a\text{ in }\R^2\setminus\{a\}.
\end{align*}
The change of gauge 
\begin{equation}\label{eq:gauge}
  u(x) \mapsto v(x):= e^{-i\Theta_a(x)}u(x) , \quad x\in \Omega\setminus\Gamma_a,    \end{equation}
maps any eigenfunction
$u\in H^{1,a}_{0}(\Omega,\C)\setminus\{0\}$ of \eqref{eq:eige_equation_a},
associated
to an eigenvalue $\lambda$, into a solution $v\in\mathcal H_a\setminus\{0\}$ of the following problem 
\begin{equation}\label{eq:eige_a}\tag{$P_a$}
    \begin{cases}
        -\Delta v = \lambda v, &\text{in }\Omega\setminus\Gamma_a,\\
        v=0, &\text{on }\partial\Omega,\\
        \gamma_+^\alpha(v) + \gamma_-^\alpha(v)=0, &\text{on }\Gamma_a,\\
        \gamma_+^\alpha(\nabla v\cdot \nu_\alpha) + \gamma_-^\alpha(\nabla v\cdot\nu_\alpha)=0,&\text{on }\Gamma_a.
    \end{cases}
\end{equation}
Since problem \eqref{eq:eige_a} has real coefficients, every
eigenspace has a basis consisting of real eigenfunctions; furthermore,
the space of real eigenfunctions of \eqref{eq:eige_a}, associated to
some eigenvalue $\lambda$, and the space of complex eigenfunctions
have the same dimension, so that the multiplicity of eigenvalues does
not change when treating \eqref{eq:eige_a} as a real problem. We can
therefore limit ourselves to considering real eigenfunctions for
\eqref{eq:eige_a}, thus treating $\mathcal H_a$-functions as
real-valued (see also \cite[Remark 3.5]{FelliNorisOgnibeneSiclari2023}
and \cite[Lemma 2.3]{BonnaillieNorisNysTerracini2014}).

Problem \eqref{eq:eige_a} is meant in the following weak sense:
$v\in \mathcal H_a$ weakly solves \eqref{eq:eige_a} if
\begin{equation*}
  v\in\widetilde{\mathcal H}_a\quad\text{and}\quad
  \int_{\Omega\setminus\Gamma_a}\nabla v\cdot
  \nabla w\,dx=\lambda\int_\Omega vw\,dx\quad\text{for all }w\in 
  \widetilde{\mathcal H}_a.
\end{equation*}
In particular, \eqref{eq:eige_equation_a} and \eqref{eq:eige_a} have
the same eigenvalues, and eigenfunctions mutually related by the gauge
transformation \eqref{eq:gauge} (if the eigenfunctions of
\eqref{eq:eige_equation_a} are chosen to satisfy
\eqref{eq:propertyP}).

For every $\alpha\in(-\pi,\pi]$, we consider the limit half-line
\begin{equation*}
\Gamma_0^\alpha=\{ t(\cos\alpha,\sin\alpha): t\in (-\infty,0] \},
\end{equation*}
noting that $\Gamma_0^\alpha=\overline{\Gamma_a\setminus S_a}$ for all
$a\in\R^2\setminus\{0\}$ belonging to the half-line from $0$ with
slope $\alpha$.  We also define $\mathcal{H}_0^\alpha$ as the closure
of
$\{ w\in H^1(\Omega\setminus \Gamma_0^\alpha): \ w=0 \text{ in a
  neighborhood of }\partial\Omega \}$ in
$H^1(\Omega\setminus \Gamma_0^\alpha)$, and
\[
  \widetilde{\mathcal H}_0^\alpha:= \{w\in \mathcal{H}_0^\alpha : \
  \gamma^\alpha_+(w) + \gamma^\alpha_-(w)=0 \text{ on }\Gamma_0^\alpha
  \},
\]
 both endowed with the norm
\[
\|w\|_{\mathcal H_0^\alpha} = \left( \int_{\Omega\setminus \Gamma_0^\alpha} |\nabla w|^2\,dx \right)^{1/2}.
\]
As $a\to0$ along the direction determined by the slope $\alpha$, the
limit eigenvalue problem is
\begin{equation}\label{eq:eige_0}\tag{$P_0^\alpha$}
    \begin{cases}
        -\Delta v = \lambda v, &\text{in }\Omega\setminus\Gamma_0^\alpha,\\
        v=0, &\text{on }\partial\Omega,\\
        \gamma_+^\alpha(v) + \gamma_-^\alpha(v)=0, &\text{on }\Gamma_0^\alpha,\\
        \gamma_+^\alpha(\nabla v\cdot \nu) + \gamma_-^\alpha(\nabla v\cdot\nu)=0,&\text{on }\Gamma_0^\alpha,
    \end{cases}
\end{equation}
meant a weak sense, i.e. $v\in \mathcal H_0^\alpha$ weakly solves \eqref{eq:eige_0} if 
\begin{equation}\label{eq:eige-0-weak}
  v\in\widetilde{\mathcal H}_0^\alpha\quad
  \text{and}\quad  \int_{\Omega\setminus\Gamma_0^\alpha}\nabla v\cdot
  \nabla w\,dx
  =\lambda\int_\Omega vw\,dx\quad\text{for all }w\in 
  \widetilde{\mathcal H}_0^\alpha.
\end{equation}
Solutions $u\in H^{1,0}_{0}(\Omega,\C)$ of $(E_0)$ correspond to those
of \eqref{eq:eige_0} through the gauge transformation
\begin{equation}\label{eq:gauge0}
G_\alpha(u)=e^{-i\Theta_0^\alpha}u,
\end{equation}
where 
\begin{equation*}
\Theta_0^\alpha: \R^2\setminus\{0\}\to \R,\quad
    \Theta_0^\alpha(r\cos t,r\sin t)=
    \begin{cases}
       \frac12 t,&\text{if }t\in[0,\pi+\alpha),\\
       \frac12 t-\pi,&\text{if }t\in[\pi+\alpha,2\pi),
    \end{cases}
\end{equation*}
i.e.
\begin{equation*}
  G_\alpha(u)(r\cos t,r\sin t)=f_\alpha(t) e^{-i\frac t2}u  (r\cos t,r\sin t),
\end{equation*}
where 
    \begin{equation}\label{eq:f-alpha}
    f_\alpha(t):=\begin{cases}
+1, & \text{if }t\in [0,\alpha+\pi),\\
        -1, & \text{if }t\in [\alpha+\pi,2\pi).        
    \end{cases}
\end{equation}
The function $f_\alpha$ is extended to be a $2\pi$-periodic function
in the whole $\R$. In particular, for every $\alpha\in (-\pi,\pi]$,
\begin{equation*}
   f_\alpha(\alpha)=
   \begin{cases}
        -1,&\text{if }\alpha\in(-\pi,0),\\
        1,&\text{if }\alpha\in[0,\pi].        
    \end{cases}
\end{equation*}
In particular \eqref{eq:eige_0} and $(E_0)$ have the same
eigenvalues. We also notice that, if $u$ is an eigenfunction of
$(E_0)$ satisfying \eqref{eq:propertyP}, then $G_\alpha(u)$ is real
valued for every $\alpha\in(-\pi,\pi]$.

The following result from \cite[Proposition
3.6]{FelliNorisOgnibeneSiclari2023} contains a description of the
behavior of any weak solution $v$ to \eqref{eq:eige_0} near the tip of
the crack $\Gamma_0^\alpha$. This is actually a reformulation of
\eqref{eq:131}, once we have performed the gauge transformation
\eqref{eq:gauge0}.
\begin{proposition}\label{p:asyeige}
  If $u\in H^{1,0}_{0}(\Omega,\C)\setminus\{0\}$ is a weak solution to
  $(E_0)$, in the sense clarified in \eqref{eq:weaksense}, for some
  $\lambda$, and $u$ satisfies \eqref{eq:propertyP}, then there exist
  $k=k(u)\in\N$ odd, $\beta=\beta(u)\in \R\setminus\{0\}$, and
  $\omega=\omega(u)\in \big[0,\frac{2\pi}{k}\big)$ such that
\begin{equation}\label{eq:asy-u-k-reale}
  r^{-k/2} u(r\cos t,r\sin t) \to 
  \beta e^{i\frac t2}
  \sin\big(\tfrac k2(t-\omega)\big) \quad \text{in }C^{1,\tau}([0,2\pi],\C)
\end{equation}
as $r\to0^+$ for all $\tau\in (0,1)$. Furthermore, if
$\alpha\in(-\pi,\pi]$ and $v=G_\alpha(u)$, then $v$ is a real-valued
nontrivial weak solution to \eqref{eq:eige_0}, in the sense clarified
in \eqref{eq:eige-0-weak}, and
        \begin{equation}\label{eq:behaviour-lim}
    r^{-k/2}v(r\cos t,r\sin t) \to \beta f_\alpha(t) \sin\big(\tfrac{k}{2} (t-\omega)\big) \qquad \text{as }r\to0^+
    \end{equation}
    in $C^{1,\tau}([0,2\pi]\setminus\{\alpha+\pi\})$, where $f_\alpha$
    is defined in \eqref{eq:f-alpha}.
\end{proposition}
We observe that the quantities $\beta=\beta(u)$, $k=k(u)$, and
$\omega=\omega(u)$ appearing in \eqref{eq:behaviour-lim} are
independent of $\alpha$ and depend only on $u$. The dependence on
$\alpha$ in the right-hand side of \eqref{eq:behaviour-lim} is only
seen in the term $f_\alpha$ (which on the other hand is independent of
$u$.).

If $u$ is an eigenfunction of $(E_0)$ satisfying \eqref{eq:propertyP}
and $\alpha\in(-\pi,\pi]$, we denote
    \begin{equation}\label{eq:psi}
      \Psi_\alpha^{u}(x)=\Psi_\alpha^u(r\cos t,r\sin t)=
      \beta f_\alpha(t) \,r^{\frac k2}\sin\big(\tfrac{k}{2} (t-\omega)\big),
    \end{equation}
    with $\beta=\beta(u)$, $k=k(u)$, and $\omega=\omega(u)$ being as in \eqref{eq:asy-u-k-reale}.

\subsection{Some key quantities}
In the asymptotic expansion for simple Aharonov-Bohm eigenvalues
derived in \cite{FelliNorisOgnibeneSiclari2023}, in the more general
framework of many coalescing poles, the dominant term is related to
the minimum of an energy functional, associated with the configuration
of poles and defined on a space of functions suitably jumping through
some cracks.  We recall now how this quantity is defined, in the
particular case of a single pole considered in the present paper,
highlighting in the notation the dependence on the pole and the slope
of motion, as we will vary them in the subsequent argument.

For every eigenfunction $v$ of problem \eqref{eq:eige_0} and
$a=|a|(\cos\alpha,\sin\alpha)\in\R^2\setminus\{0\}$, we define the
linear functional $L_a^v: \ \mathcal H_a \to \R$ as
\[
L_a^v(w):= 2\int_{S_a} \nabla v\cdot \nu_\alpha \,\gamma_+^\alpha(w)\,dS,
\]
and $J_a^v: \ \mathcal H_a \to \R$ as  
\[
J_a^v(w):= \frac12 \int_{\Omega\setminus\Gamma_a} |\nabla w|^2\,dx +  L_a^v(w).
\]
As observed in \cite[Proposition 4.2]{FelliNorisOgnibeneSiclari2023},
for every eigenfunction $u$ of $(E_0)$ satisfying \eqref{eq:propertyP}
and every $a\in \R^2\setminus\{0\}$ belonging to the half-line from
$0$ with slope $\alpha$, there exists a unique $U_a^u\in \mathcal H_a$
such that
\begin{equation}\label{eq:def-potenziali}
     U^u_a-G_\alpha(u)\in \widetilde{\mathcal H}_a\quad\text{and}\quad 
     J_a^{G_\alpha(u)}(U^u_a)=\mathcal E_a^u,
\end{equation}
where 
\begin{equation}\label{eq:def-E-u-a}
\mathcal E_a^u:=\min\left\{
    J_a^{G_\alpha(u)}(w):w\in \mathcal H_a\text{ and }w-G_\alpha(u)\in \widetilde{\mathcal H}_a\right\}.
\end{equation}
After scaling and blow-up, the limit problem is described in terms of
cracks, spaces and functionals depending only on the direction
$\alpha$. For every $\alpha\in (-\pi,\pi]$, we define
\begin{align*}
    &\Gamma^\alpha_1=\Gamma_{(\cos\alpha,\sin\alpha)}=\{ t(\cos\alpha,\sin\alpha): t\in (-\infty,1] \},\\
&S^\alpha_1=S_{(\cos\alpha,\sin\alpha)}=\{ t(\cos\alpha,\sin\alpha): t\in [0,1] \}.
\end{align*}
Denoting    $D_r=\{x\in \R^2:|x|<r\}$ for every $r>0$, we then introduce the functional space
\begin{equation*}
\widetilde X_\alpha:= \left\{
\begin{array}{ll}
  \!\!w\in L^1_{\rm loc}(\R^2):\!\!\! &w\in H^1(D_r\setminus \Gamma_1^\alpha) \text{ for all }r>0, \\  [5pt]
                                      &\nabla w\in L^2(\R^2\setminus
                                        \Gamma_1^\alpha,\R^2),\
                                        \gamma_+^\alpha(w) + \gamma_-^\alpha(w)=0 \text{ on }\Gamma_0^\alpha
\end{array}\!\!
\right\},
\end{equation*}
endowed with the norm 
\begin{equation*}
    \|w\|_{\widetilde X_\alpha}=\left(\int_{\R^2\setminus \Gamma_1^\alpha}|\nabla w|^2\,dx\right)^{\!1/2},
\end{equation*}
and its closed subspace
\begin{equation*}
    \widetilde{\mathcal H}_\alpha=\{w\in \widetilde X_\alpha:
    \gamma_+^\alpha(w) + \gamma_-^\alpha(w)=0 \text{ on }S_1^\alpha\}.
\end{equation*}
For every $\alpha\in (-\pi,\pi]$ and  every eigenfunction $u$ of $(E_0)$ satisfying \eqref{eq:propertyP},
we consider the quadratic functional 
\begin{equation}\label{d:Jtilde}
  \widetilde{J}_\alpha^u:\widetilde X_\alpha\to\R,\quad
  \widetilde{J}_\alpha^u(w)=\frac12\int_{\R^2\setminus\Gamma^\alpha_1}|\nabla w|^2\,dx+\widetilde{L}_\alpha^u(w),
\end{equation}
where 
\begin{equation}\label{d:Ltilde}
    \widetilde{L}_\alpha^u:\widetilde X_\alpha\to\R,\quad
    \widetilde{L}_\alpha^u(w)
    =2\int_{S^\alpha_1}\nabla \Psi^u_\alpha\cdot \nu_\alpha \gamma^\alpha_+(w)\,dS.
\end{equation}
As proved in \cite[Proposition 6.4]{FelliNorisOgnibeneSiclari2023},
there exists a unique $\widetilde{U}_\alpha^u\in \widetilde X_\alpha$
such that
\[
\widetilde{U}_\alpha^u- \eta \Psi_\alpha^u \in \widetilde{\mathcal H}_\alpha
\quad\text{and}\quad \widetilde{J}_\alpha^u (\widetilde{U}_\alpha^u) = \widetilde{\mathcal E}_\alpha^{u},
\]
where 
\begin{equation}\label{d:Etilde}
\widetilde{\mathcal E}_\alpha^{u}:=    \min\{
\widetilde{J}_\alpha^u(w):
\ w\in \widetilde X_\alpha\text{ and } w- \eta \Psi_\alpha^u \in\widetilde{\mathcal H}_\alpha \},
\end{equation}
being 
$\eta\in C^\infty_{\rm c}(\R^2)$  a fixed radial cut-off function 
such that
\begin{equation*}
        \begin{cases} 
        0\le \eta(x) \le 1 \text{ for all $x \in \R^2$},\\ 
    \eta(x)=1 \text{ if $x\in D_{1}$},\quad \eta(x)=0 \text{ if $x\in \R^2 \setminus D_{2}$},\\
    |\nabla \eta|\leq 2 \text{ in $D_{2}\setminus D_{1}$}.
\end{cases}
\end{equation*}

We recall from \cite{FelliNorisOgnibeneSiclari2023} the
following key results. The first one asserts the negligibility of the
$L^2$ mass with respect to the energy of the functions $U_a^u$, as the
pole approaches the origin, whereas the second one is a blow-up result
for both the functions $U_a^u$ and the quantities $\mathcal E_a^u$.
We note that the simplicity assumption required throughout the paper
\cite{FelliNorisOgnibeneSiclari2023} is not, however, used to prove
these two specific results, which therefore hold for any eigenfunction
$u$ of $(E_0)$ satisfying \eqref{eq:propertyP}, possibly associated
with a multiple eigenvalue.

\begin{lemma}{\rm (\cite[Proposition 4.6]{FelliNorisOgnibeneSiclari2023})}\label{l:prop4.6FNOS}
Let   $u$ be an eigenfunction of $(E_0)$ satisfying the property \eqref{eq:propertyP}. 
 If $\alpha\in (-\pi,\pi]$ and $a=|a|(\cos\alpha,\sin\alpha)$, then
 \[
 \int_{\Omega} |U_a^u|^2\,dx = o(\|U_a^u\|_{\mathcal H_a}^2) \quad \text{as }|a|\to0.
 \]
\end{lemma}\begin{theorem}{\rm (\cite[Theorem 2.2 and Proposition
    6.7]{FelliNorisOgnibeneSiclari2023})}\label{t:teo2.2FNOS}
  Let $u$ be an eigenfunction of $(E_0)$ satisfying
  \eqref{eq:propertyP} and having at 0 a zero or order $k/2$ for some
  $k\in\N$ odd. Let $\alpha\in(-\pi,\pi]$ and
  $a=|a|(\cos\alpha,\sin\alpha)$. Then
 \begin{align*}
     &\lim_{|a|\to0} |a|^{-k} \mathcal E_a^u = \widetilde{\mathcal E}_\alpha^{u},\\
     &\lim_{|a|\to0} |a|^{-k} \big(\mathcal E_a^u -
       L_a^{G_\alpha(u)}(G_\alpha(u))\big)=
       \widetilde{\mathcal E}^u_\alpha- \widetilde{L}^u_\alpha(\Psi_\alpha^u),\\
     &\|U^u_a\|_{\mathcal H_a}=O(|a|^{k/2})\quad\text{as }|a|\to0,\\
     &\lim_{|a|\to0} |a|^{-k/2}U^u_a(|a|\cdot)=\widetilde{U}_\alpha^u\quad\text{in }\widetilde X_\alpha,
 \end{align*}
where $U^u_a$ is extended trivially in $\R^2\setminus\Omega$.
\end{theorem}

\section{Perturbation theory around the origin along a fixed direction}\label{sec:appLemma}

In order to study the genericity properties of simple Aharonov--Bohm
eigenvalues, we start analyzing the asymptotic behavior of
eigenbranches departing from a fixed eigenvalue $\lambda_N^0$ of
$(E_0)$ (which is consequently also an eigenvalue of
\eqref{eq:eige_0}), as the perturbed pole is moving along a fixed
  direction $(\cos\alpha,\sin\alpha)$, towards the limit pole located
  at $0$. The strategy will be then studying how the coefficient
  of the dominant term depends on the angle $\alpha$, in order to
identify the directions for which the branches have different
  leading terms (either in the order or in the coefficient), thus
  remaining detached from each other.

In general, we assume that 
\begin{align}\label{eq:multiple}
&\text{the multiplicity of $\lambda_N^0$
(both as an eigenvalue of $(E_0)$}\\
\notag & \text{and as an eigenvalue of \eqref{eq:eige_0})
is $m\ge1$}.
\end{align}
We choose the index $N$ in such a way that 
\begin{equation}\label{eq:index}
  \lambda_N^0=  \lambda_{N+i-1}^0\quad\text{for all }i\in\{1,\dots,m\}.
\end{equation}
We consider the associated $K$-real eigenspace 
\begin{equation}\label{eq:Elambda0}
    E(\lambda_N^0)=\{\varphi\in H^{1,0}_{0}(\Omega,\C): \varphi \text{
      weakly solves }(E_0)
    \text{ with $\lambda=\lambda_N^0$ and satisfies \eqref{eq:propertyP}}\} 
\end{equation}
conceived as a real vector space, which, under assumption
\eqref{eq:multiple}, turns out to have dimension $m$.

A first asymptotic expansion along a fixed direction $\alpha$ is
contained in the following result.
\begin{theorem}\label{thm:approxEVs} 
Let $\alpha\in(-\pi,\pi]$ and $a=|a|(\cos\alpha,\sin\alpha)\in \R^2\setminus\{0\}$. 
For every $i\in\{1,\dots,m\}$, the eigenbranches can be expanded as
\begin{equation}\label{eq:asymptEV}
	\lambda_{N+i-1}^a=\lambda_N^0+\mu_i^a+o(\chi_a^2) \mbox{ as }|a|\to0,
\end{equation}
where
\begin{equation}\label{eq:chiEps}
  \chi_a:=\sup\left\{\| U_a^u\|_{\mathcal H_a}+\sup_{S_a}|u| :
    \,u\in E(\lambda_N^0)\mbox{ and } \|u\|_{L^2(\Omega,\C)}=1\right\}
\end{equation}
and $\{\mu_i^a\}_{i=1}^m$ are the eigenvalues (taken in non-decreasing
order) of the bilinear form
\begin{align}\label{eq:def-r_a}
    r_a(u,w) &:=2 \int_{S_a} (\nabla
               G_\alpha(u)\cdot\nu_\alpha)\gamma_+^\alpha(U_a^w)\,dS -
               2 \int_{S_a} (\nabla G_\alpha(u)\cdot  \nu_\alpha)G_\alpha(w)\,dS\\
   \notag &\quad + 2 \int_{S_a} (\nabla
            G_\alpha(w)\cdot\nu_\alpha)\gamma_+^\alpha(U_a^u)\,dS
            -2 \int_{S_a} (\nabla G_\alpha(w)\cdot    \nu_\alpha)G_\alpha(u) \,dS\\
        \notag&\quad+  \int_{\Omega\setminus\Gamma_a}\nabla
                U_a^u\cdot\nabla U_a^w\,dx
                - \lambda_N^0\int_\Omega U_a^u U_a^w\,dx , 
\end{align}
defined for $u,w\in E(\lambda_N^0)$, 
where $U_a^u,U_a^w\in \mathcal H_a$ are as in \eqref{eq:def-potenziali}--\eqref{eq:def-E-u-a}.
\end{theorem}

Although Theorem \ref{thm:approxEVs} provides a good approximation for
perturbed eigenvalues when the limit one is simple, it is not
exhaustive for multiple ones. For instance, for some $i$ it may be
$\mu_i^a=o(\chi_a^2)$ as $|a|\to0$: in this case \eqref{eq:asymptEV}
reduces to $\lambda_{N+i-1}^a-\lambda_N=o(\chi_a^2)$, providing no
sharp quantification of the vanishing order, but just an estimate.

From Theorem \ref{t:teo2.2FNOS} and Lemma \ref{l:prop4.6FNOS} it
follows that, for every $u\in E(\lambda_N^0)\setminus\{0\}$ having at
$0$ a zero of order $k/2$ for some $k\in\N$ odd,
$\alpha\in (-\pi,\pi]$, and $a=|a|(\cos\alpha,\sin\alpha)$,
\begin{align}\label{eq:rauu}
    \lim_{|a|\to0}|a|^{-k}r_a(u,u)&=
                                    \lim_{|a|\to0}|a|^{-k}\left(2\left(J_a^{G_\alpha(u)}(U^u_\alpha)
                                    -L^{G^\alpha(u)}(G_\alpha(u))\right)
                                    -\lambda_N^0\int_\Omega|U^u_a|^2\,dx\right)\\
    \notag&=2\left(\widetilde{\mathcal E}_\alpha^{u}-\widetilde{L}^u_\alpha(\Psi_\alpha^u)\right).
\end{align}
For every $\alpha\in (-\pi,\pi]$ and $u\in E(\lambda_N^0)\setminus\{0\}$, we define 
\begin{equation}\label{d:functionC}
    \mathcal C(\alpha,u)=2\Big(\widetilde{\mathcal E}_\alpha^{u}-\widetilde{L}^u_\alpha(\Psi_\alpha^u)\Big).
\end{equation}
The function $\alpha\mapsto \mathcal C(\alpha,u)$ is understood to be
extended with periodicity $2\pi$ to the whole $\R$.

We refer to Section \ref{sec:functionC} for the study of the
properties of the function $\mathcal C(\alpha,u)$. We devote the rest
of this section to the proof of Theorem \ref{thm:approxEVs}.  For
this purpose, let us assume that the operator's pole $a$ is moving
towards $0$ along a fixed direction $\alpha \in(-\pi,\pi]$, i.e.
\[
a=|a|(\cos\alpha,\sin\alpha) \quad \text{and}\quad |a|\to0.
\]
In order to find an approximation of the perturbed eigenvalues
$\{\lambda_j^a\}$, we use a slight modification of a lemma from
G. Courtois \cite{Courtois1995}, itself based on the work of Y. Colin
de Verdiére \cite{ColindeV1986}.  We refer to \cite[Appendix]{ALM2022}
for a detailed proof.

We recall from \cite[Proposition 4.2]{FelliNorisOgnibeneSiclari2023}
that the unique function $U_a^u\in \mathcal H_a$ which attains the
minimum $\mathcal E_a^u$ in \eqref{eq:def-E-u-a} satisfies
\begin{equation}\label{eq:weakvau}
\begin{cases}
  U_a^u - G_\alpha(u) \in \widetilde{\mathcal H}_a\\[3pt]
  {\displaystyle{\int_{\Omega\setminus\Gamma_a}}} \nabla U_a^u \cdot
  \nabla w\,dx = -2 {\displaystyle{\int_{S_a}}} \nabla
  G_\alpha(u)\cdot \nu_\alpha \,\gamma_+^\alpha(w)\,dS \qquad
  \text{for all }w\in \widetilde{\mathcal H}_a.
\end{cases}
\end{equation}
We consider the quantity $\chi_a$ defined in \eqref{eq:chiEps}.  In
view of \eqref{eq:asy-u-k-reale} and \cite[Proposition
4.4]{FelliNorisOgnibeneSiclari2023}, we have the following preliminary
result.
\begin{lemma}\label{l:error}
If $\alpha\in (-\pi,\pi]$ and $a=|a|(\cos\alpha,\sin\alpha)$, then 
	\[
		\chi_a\to 0\quad\text{as }|a|\to 0.
	\]
\end{lemma} 
We denote by $\Pi_a$ the linear map
\begin{equation*}
\begin{array}{crcl}
	\Pi_a:& E(\lambda_N^0)& \to     & \mathcal H_a\\
			 & u            & \mapsto & G_\alpha(u)-U_a^{u},
\end{array}
\end{equation*}
where $\mathcal H_a$ is considered as a subspace of $L^2(\Omega)$ and  
$E(\lambda_N^0)$ as a subspace of 
\begin{equation*}
L^2_K(\Omega):=\{u\in L^2(\Omega,\C):u\text{ satisfies \eqref{eq:propertyP}}\},    
\end{equation*}
here meant as a Hilbert space over $\R$ endowed with the scalar
product $(u,v)_{L^2_K}=\int_\Omega u\overline{v}\,dx$ (which takes
real values if $u,v\in L^2_K(\Omega)$). We observe that, in view of
\eqref{eq:weakvau},
\begin{equation}\label{eq:subspHtilde}
    \Pi_a(E(\lambda_N^0))\subseteq \widetilde{\mathcal H}_a.
\end{equation}
\begin{lemma} \label{l:norm}
	Let $\alpha\in (-\pi,\pi]$ and $a=|a|(\cos\alpha,\sin\alpha)$.
 If $M_a:=\left\|G_\alpha -\Pi_a \right\|_{{\mathcal L(E(\lambda_N^0),L^2(\Omega))}}$, then
	 \[
  M_a=o(\chi_a)\quad\text{as }|a|\to 0.
  \]
\end{lemma}
\begin{proof} 
  Let $u\in E(\lambda_N^0)$ be such that
  $\|u\|_{L^2_K(\Omega)}=\|u\|_{L^2(\Omega,\C)}=1$. Let us write
  $u=\sum_{i=1}^m c_i u_{N+i-1}$, for some real numbers
  $\{c_i\}_{i=1}^m$ such that $\sum_{i=1}^m c_i^2=1$, being
  $\{u_{N+i-1}\}_{i=1}^m$ a basis of $E(\lambda_N^0)$ orthonormal in
  $L^2_K(\Omega)$. By definition, we have
  $(G_\alpha-\Pi_a)u=U_a^u$. Hence, by linearity and Cauchy-Schwarz's
  inequality we find that
 \begin{align*}
   \|(G_\alpha-\Pi_a)u\|_{L^2(\Omega)} &= \|U_a^u\|_{L^2(\Omega)} \\
                                       &\le \sum_{i=1}^m
                                         |c_i|\|U^{u_{N+i-1}}_{a
                                         }\|_{L^2(\Omega)}
                                         \le \bigg(\sum_{i=1}^m
                                         c_i^2\bigg)^{\frac12}\left(\sum_{i=1}^m
                                         \|U^{u_{N+i-1}}_{a }\|^2_{L^2(\Omega)}\right)^{\frac12}\\
                                       &=\left(\sum_{i=1}^m
                                         \|U_a^{u_{N+i-1}}\|_{\mathcal
                                         H_a}^2
                                         \frac{\|U_a^{u_{N+i-1}}\|^2_{L^2(\Omega)}}
                                         {\| U_a^{u_{N+i-1}}\|_{\mathcal H_a}^2}\right)^{\frac12}\\ 
                                       &\le
                                         \sqrt{m}\,\chi_a\,\max_{1\le
                                         i\le m}
                                         \frac{\|U_a^{u_{N+i-1}}\|_{L^2(\Omega)}}{\|U_a^{u_{N+i-1}}\|_{\mathcal H_a}}.
 \end{align*}
 According to Lemma \ref{l:prop4.6FNOS} the last term is $o(\chi_a)$
 as $|a|\to0$ and this concludes the proof.
\end{proof}

In particular, Lemma \ref{l:norm} implies that $M_a<1$, so that
$\Pi_a$ is injective, for $|a|$ small enough. We will always assume
this to be the case in the rest of this section.

We now refer to a result in \cite{ALM2022}, which is a generalization
of a lemma by Colin de Verdiére \cite{ColindeV1986}. It establishes
the possibility of describing the variation of the perturbed
eigenvalues by means of a suitable quantity. The latter is given by
the control of the quadratic form that measures the variation.

\begin{proposition}[{\cite[Proposition B.1]{ALM2022}}]\label{p:appEV}
  Let $(\mathfrak H, \|\cdot\|)$ be a real Hilbert space and
  $q:\mathfrak D\times \mathfrak D\to\R$ be a symmetric bilinear form,
  with the domain $\mathfrak D$ being a dense subspace of
  $\mathfrak H$, such that
\begin{itemize}
    \item[-]  $q$ is semi-bounded from below (not necessarily positive), 
i.e. there exists $c\in\R$ such that 
$q(u,u)\geq c\|u\|^2$ for all $u\in \mathfrak D$;
\item[-] $q$ has an increasing sequence of eigenvalues  $\{ \nu_i \}_{i\geq1}$;
\item[-] there exists an orthonormal basis $\{ g_i \}_{i\geq1}$of
  $\mathfrak H$ such that $g_i\in\mathfrak D$ is an eigenvector of $q$
  associated to the eigenvalue $\nu_i$, i.e.  $q(g_i,v)=\nu_i(g_i,v)$
  for all $i\geq1$ and $v\in \mathfrak D$, being $(\cdot,\cdot)$ the
  scalar product in $\mathfrak H$.
\end{itemize}
Let $N$ and $m$ be positive integers, $F$ a $m$-dimensional subspace
of $\mathfrak D$ and $\{ \xi_i^F\}_{i=1}^m$ the eigenvalues of the
restriction of $q$ to $F$ (in ascending order).

Assume that there exist positive constants $\gamma$ and $\delta$ such that
\begin{itemize}
 \item[(H1)] $ 0<\delta<\gamma/\sqrt2$;
 \item[(H2)] for all $i\in\{1,\dots,m\}$, $|\nu_{N+i-1}|\le\gamma$,
   $\nu_{N+m}\ge \gamma$ and, if $N\ge2$, $\nu_{N-1}\le-\gamma$;
 \item[(H3)] $|q(\varphi,g)|\leq \delta\, \|\varphi \|\,\|g\|$ for all
   $g\in\mathfrak D$ and $\varphi \in F$.
\end{itemize}
Then we have
\begin{itemize}
 \item[(i)] $\left|\nu_{N+i-1}- \xi_i^F \right|\le\frac{ 4}{\gamma}\delta^2$ for all $i=1,\ldots,m$; 
 \item[(ii)]
   $\left\| \Pi_N - \mathbb{I}\right\|_{\mathcal L(F,\mathfrak H)}
   \leq { \sqrt 2}\delta/\gamma$, where $\Pi_N$ is the projection onto
   the subspace of $\mathfrak D$ spanned by
   $\{g_N,\ldots,g_{N+m-1}\}$.
\end{itemize}
\end{proposition}

\subsection{Application to our setting}\label{subsec:application-our}
We are going to apply Proposition \ref{p:appEV} in the following
way. For $a=|a|(\cos\alpha,\sin\alpha)$, with $\alpha\in(-\pi,\pi]$
fixed and $|a|>0$ small enough, we introduce the following set of
definitions \eqref{notfirst}--\eqref{notlast}:
\begin{align}
 &\mathfrak H_a := L^2(\Omega)~\text{with norm }\|{\cdot}\|:=\|{\cdot}\|_{L^2(\Omega)};
 \quad \mathfrak D_a :=\widetilde{\mathcal H}_a;
 \label{notfirst}\\
  \label{eq:def-q-a}&q_a(u,v):= \int_{\Omega\setminus\Gamma_a} \nabla
                      u\cdot\nabla v\dx-
                      \lambda_N^0\int_{\Omega} uv\dx,\quad\text{for every }u,v\in\mathfrak D_a;\\
  &F_{a}:= \Pi_a(E(\lambda_N^0)). \label{notlast}
\end{align}
By construction and the spectral equivalence between problems
\eqref{eq:eige_a} and \eqref{eq:eige_equation_a}, the eigenvalues of
$q_a$ are $\{\lambda_i^a-\lambda_N^0\}_{i\ge1}$. We use the notation
$\nu_i^a:=\lambda_i^a-\lambda_N^0$. If $|a|$ is small enough, Lemma
\ref{l:norm} implies that $\Pi_a$ is injective, so that $\Pi_a$ is
bijective from $E(\lambda_N^0)$ onto $F_a$ and $F_a$ is proved to be a
$m$-dimensional subspace of $\widetilde{\mathcal H}_a$, see
\eqref{eq:subspHtilde}. Since $\lambda_i^a\to\lambda_i^0$ as $a\to0$
for all $i\in\N\setminus\{0\}$ and $\lambda_N^0$ is of multiplicity
$m$, the assumption (H2) in Proposition \ref{p:appEV} is fulfilled for
$|a|>0$ small enough if we take, for instance,
\begin{equation*}
	\gamma:={\frac12}\min\{\lambda_N^0-\lambda_{N-1}^0,\lambda_{N+m}^0-\lambda_{N+m-1}^0\}
\end{equation*}
when $N\ge2$, whereas $\gamma:={ \frac12}
	\left(\lambda_{1+m}^0-\lambda_{1}^0\right)$ when $N=1$.

        It remains to check whether condition (H3) in Proposition
        \ref{p:appEV} is satisfied. Let us choose $v\in F_a$ and
        $w\in \widetilde{\mathcal H}_a$. Since $\Pi_a$ is injective by
        Lemma \ref{l:norm}, there exists a unique $u\in E(\lambda_N)$
        such that $v=\Pi_a u$. Hence, we have
\begin{align}\label{eq:biH31}
  q_a&(v, w) =q_a(G_\alpha(u)-U_a^u, w) 
       = \int_{\Omega\setminus\Gamma_a} \nabla (G_\alpha(u)-U_a^u)
       \cdot \nabla w\,dx
       - \lambda_N^0 \int_{\Omega} (G_\alpha(u)-U_a^u)w\,dx\\
  \notag&=\int_{\Omega\setminus\Gamma_a} \nabla G_\alpha(u) \cdot
          \nabla w\,dx
          - \lambda_N^0 \int_{\Omega} G_\alpha(u)w\,dx-
          \int_{\Omega\setminus\Gamma_a} \nabla U_a^u \cdot \nabla w\,dx  + \lambda_N^0 \int_{\Omega} U_a^uw\,dx.
\end{align}
On one hand, we have
\begin{equation}\label{eq:firstH3}
  \int_{\Omega\setminus\Gamma_a} \nabla G_\alpha(u) \cdot \nabla w\,dx -
  \lambda_N^0 \int_{\Omega} G_\alpha(u)w\,dx =
  -2 \int_{S_a} (\nabla G_\alpha(u)\cdot \nu_\alpha) \gamma_+^\alpha(w)\,dS.
\end{equation}
In order to prove \eqref{eq:firstH3}, we test \eqref{eq:eige_0}
(which is satisfied by $G_\alpha(u)$ with $\lambda=\lambda_N^0$)
with $w$, and integrate it by parts over
$\Omega\setminus\Gamma_a$. Denoting $\nu_{ext}$ as the external normal
vector with $|\nu_{ext}|=1$, we obtain
\begin{align*}
  \lambda_N^0\int_{\Omega\setminus\Gamma_a} G_\alpha(u)w\,dx& = \int_{\Omega\setminus\Gamma_a}\nabla G_\alpha(u)\cdot \nabla w\,dx - \int_{\partial(\Omega\setminus\Gamma_a)} (\nabla G_\alpha(u)\cdot \nu_{ext})w\,dS \\
                                                            &= \int_{\Omega\setminus\Gamma_a}\nabla G_\alpha(u)\cdot \nabla w\,dx + \int_{\Gamma_0^\alpha} \gamma_+^\alpha (\nabla G_\alpha(u)\cdot \nu_\alpha)\left( 
                                                              \gamma_+^\alpha(w) + \gamma_-^\alpha(w) \right)\,dS \\
                                                            &\quad + \int_{S_a} (\nabla G_\alpha(u)\cdot \nu_\alpha) \left( 
                                                              \gamma_+^\alpha(w) - \gamma_-^\alpha(w) \right)\,dS \\
                                                            &= \int_{\Omega\setminus\Gamma_a}\nabla G_\alpha(u)\cdot \nabla w\,dx  + 2\int_{S_a} (\nabla G_\alpha(u)\cdot \nu_\alpha) \gamma_+^\alpha(w)\,dS
\end{align*}
since $w\in \widetilde{\mathcal H}_a$.  Thus, \eqref{eq:firstH3} is
proved. In view of \eqref{eq:biH31}, \eqref{eq:weakvau}, and
\eqref{eq:firstH3} we reach
\begin{align*}
|q_a(G_\alpha(u)-U_\alpha^u, w)|& = \left|\lambda_N^0\int_{\Omega} U_a^uw\,dx\right| \\
&\le \lambda_N^0\|U_a^u\|\|w\| \le \lambda_N^0 M_a \|G_\alpha(u)\|\|w\|\le\lambda_N^0\frac{M_a}{1-M_a}\|v\|\|w\|,
\end{align*}
because
$\|U_a^u\| = \|G_\alpha(u)-(G_\alpha(u)-U_a^u)\|= \| (G_\alpha -
\Pi_a)u\|\leq M_a\|u\|_{L^2_K(\Omega)}=M_a\|G_\alpha(u)\|$ and
$\|G_\alpha(u)\|=\|(G_\alpha - \Pi_a) u + \Pi_a u\| \le
M_a\|G_\alpha(u)\| + \|v\|$, so that
\[
 \|G_\alpha(u)\| \le \dfrac{\|v\|}{1-M_a}.
\]
Lemma \ref{l:norm} then implies
\begin{equation*}
	|q_a(v,w)|\leq \delta_a\|v\|\|w\|,
\end{equation*}
for some $\delta_a>0$ such that $\delta_a=o(\chi_a)$ as $|a|\to 0$.
We can now apply Proposition \ref{p:appEV} with $\delta=\delta_a$,
which implies that, for every $1\le i\le m$,
\begin{equation}\label{eq:asympt_proof}
	\lambda_{N+i-1}^a=\lambda_N^0+\xi_i^a +o(\chi_a^2)\quad\text{as }|a|\to0,
\end{equation}
where $\{ \xi_i^a \}_{i=1}^m$  are the eigenvalues (in ascending order)
of the restriction of $q_a$ to $F_a$.

\subsection{Analysis of the restricted quadratic form.}

In order to study $\{ \xi_i^a \}_{i=1}^m$, we investigate how the
bilinear form $q_a$ acts when it is restricted to the $m$-dimensional
subspace $F_a= \Pi_a(E(\lambda_N^0))$, still endowed with the
$L^2(\Omega)$-norm.

Let $r_a$ be  the bilinear form on $E(\lambda_N^0)$ defined in \eqref{eq:def-r_a}.

\begin{lemma} \label{l:restrict}
For every $\varphi,\psi\in E(\lambda_N^0)$, 
\begin{align*}
q_a\left(\Pi_a \varphi,\Pi_a \psi\right) = r_a(\varphi,\psi).
\end{align*}
\end{lemma}
\begin{proof}
     We have
    \begin{align}\label{eq:torestricted1}
      q_a&\left(\Pi_a \varphi,\Pi_a \psi\right)=  q_a(G_\alpha(\varphi)-U^\varphi_a,
           G_\alpha(\psi)-U^\psi_a) \\
         &        = \int_{\Omega\setminus\Gamma_a} \nabla (G_\alpha(\varphi) - U_a^\varphi)\cdot\nabla (G_\alpha(\psi) - U_a^\psi)\,dx - \lambda_N^0\int_\Omega
           (G_\alpha(\varphi) - U_a^\varphi)(G_\alpha(\psi) - U_a^\psi)\,dx\notag\\
         &= \int_{\Omega\setminus\Gamma_a}\nabla G_\alpha(\varphi)\cdot\nabla G_\alpha(\psi)\,dx - \int_{\Omega\setminus\Gamma_a}\nabla G_\alpha(\psi)\cdot\nabla U^\varphi_a\,dx - \int_{\Omega\setminus\Gamma_a}\nabla G_\alpha(\varphi)\cdot\nabla U^\psi_a\,dx \notag\\
         &\quad+  \int_{\Omega\setminus\Gamma_a}\nabla U^\varphi_a\cdot\nabla U^\psi_a\,dx - \lambda_N^0\int_\Omega G_\alpha(\varphi)G_\alpha(\psi)\,dx + \lambda_N^0\int_\Omega G_\alpha(\varphi) U^\psi_a\,dx \notag\\
         &\quad + \lambda_N^0\int_\Omega G_\alpha(\psi)U^\varphi_a\,dx - \lambda_N^0\int_\Omega U^\varphi_a U^\psi_a\,dx  \notag\\
      \notag &=- \int_{\Omega\setminus\Gamma_a}\nabla G_\alpha(\psi)\cdot\nabla U^\varphi_a\,dx - \int_{\Omega\setminus\Gamma_a}\nabla G_\alpha(\varphi)\cdot\nabla U^\psi_a\,dx +  \int_{\Omega\setminus\Gamma_a}\nabla U^\varphi_a\cdot\nabla U^\psi_a\,dx\notag\\
      \notag &\quad - \lambda_N^0\int_\Omega U^\varphi_a U^\psi_a\,dx   + \lambda_N^0\int_\Omega G_\alpha(\varphi) U^\psi_a\,dx  + \lambda_N^0\int_\Omega G_\alpha(\psi)U^\varphi_a\,dx, 
    \end{align}
    because $G_\alpha(\varphi)$ is an eigenfunction of
    \eqref{eq:eige_0} relative to $\lambda_N^0$. On the other hand, we
    have
\begin{multline}\label{eq:torestricted2}
  \int_{\Omega\setminus\Gamma_a} \nabla G_\alpha(\varphi)\cdot\nabla
  U^\psi_a\,dx - \lambda_N^0\int_\Omega G_\alpha(\varphi) U^\psi_a\,dx
  \\= -2 \int_{S_a} (\nabla G_\alpha(\varphi)\cdot
  \nu_\alpha)\gamma_+^\alpha(U^\psi_a)\,dS + 2 \int_{S_a} (\nabla
  G_\alpha(\varphi)\cdot \nu_\alpha)G_\alpha(\psi)\,dS.
\end{multline}
In order to prove \eqref{eq:torestricted2}, we consider
\eqref{eq:eige_0} multiplied by $U^\psi_a$ and we integrate by parts
over $\Omega\setminus\Gamma_a$. We follow the same computation as in
the proof of \eqref{eq:firstH3}, taking into account that
$U^\psi_a-G_\alpha(\psi)\in \widetilde{\mathcal H}_a$, and hence
$\gamma_+^\alpha(U^\psi_a)+\gamma_-^\alpha(U^\psi_a)=2G_\alpha(\psi)$
on $S_a$.  In a similar way,
\begin{multline}\label{eq:torestricted2-bis}
  \int_{\Omega\setminus\Gamma_a} \nabla G_\alpha(\psi)\cdot\nabla
  U^\varphi_a\,dx - \lambda_N^0\int_\Omega G_\alpha(\psi)
  U^\varphi_a\,dx \\= -2 \int_{S_a} (\nabla G_\alpha(\psi)\cdot
  \nu_\alpha)\gamma_+^\alpha(U^\varphi_a)\,dS + 2 \int_{S_a} (\nabla
  G_\alpha(\psi)\cdot \nu_\alpha)G_\alpha(\varphi)\,dS.
\end{multline}
Combining \eqref{eq:torestricted1}, \eqref{eq:torestricted2}, and
\eqref{eq:torestricted2-bis} we finally obtain
\begin{align*}
  q_a\left(\Pi_a \varphi,\Pi_a \psi\right)
  &= 2 \int_{S_a} (\nabla G_\alpha(\varphi)\cdot    \nu_\alpha)\gamma_+^\alpha(U^\psi_a)\,dS -2 \int_{S_a} (\nabla G_\alpha(\varphi)\cdot    \nu_\alpha)G_\alpha(\psi)\,dS\\
  &\quad + 2 \int_{S_a} (\nabla G_\alpha(\psi)\cdot    \nu_\alpha)\gamma^\alpha_+(U^\varphi_a)\,dS -2 \int_{S_a} (\nabla G_\alpha(\psi)\cdot    \nu_\alpha)G_\alpha(\varphi)\,dS \\
  &\quad+  \int_{\Omega\setminus\Gamma_a}\nabla U^\varphi_a\cdot\nabla U^\psi_a\,dx - \lambda_N^0\int_\Omega  U^\varphi_a U^\psi_a\,dx, 
\end{align*}
and the proof is completed.
\end{proof}

\begin{remark}\label{r:entries-r-a}
  From Lemma \ref{l:error}, Lemma \ref{l:prop4.6FNOS},
  and\cite[Propositions 3.7. and 4.1]{FelliNorisOgnibeneSiclari2023}
  it follows that, for every $u,w\in E(\lambda_N^0)$ such that
  $\|u\|_{L^2(\Omega,\C)}=\|w\|_{L^2(\Omega,\C)}=1$, we have
  $r_a(u,w)=o(\chi_a)$ as $|a|\to0$.
\end{remark}

\begin{lemma}\label{l:mu-xi}
  Let $\{\xi_j^a\}_{j=1}^m$ and $\{\mu_j^a\}_{j=1}^m$ be the
  eigenvalues (in ascending order) of the restriction of the form
  $q_a$ defined in \eqref{eq:def-q-a} to the space $F_a$ in
  \eqref{notlast} and of the form $r_a$ on
  $E(\lambda_N^0)\times E(\lambda_N^0)$ introduced in
  \eqref{eq:def-r_a}, respectively. Then, for every
  $j\in\{1,\dots,m\}$,
\begin{equation}\label{eq:exp-mu-xi}
    \xi_j^a = \mu_j^a + o(\chi_a^2) \qquad \text{as }|a|\to0,
\end{equation}
where $\chi_a$ is defined in \eqref{eq:chiEps}.
\end{lemma}
\begin{proof}
  Let $\{\varphi_1,\dots,\varphi_m\}$ be an orthonormal basis of
  $E(\lambda_N^0)$. Then the eigenvalues $\{\mu_j^a\}_{j=1}^m$
  coincide with the eigenvalues of the $m\times m$ symmetric real
  matrix
    \begin{equation*}
        \mathsf{R}_a=\left(r_a(\varphi_j,\varphi_k)\right)_{jk}.
    \end{equation*}
    In view of Remark \ref{r:entries-r-a},
    \begin{equation}\label{eq:mu-xi-1}
        r_a(\varphi_j,\varphi_k)=o(\chi_a)\quad\text{as }|a|\to0,
    \quad \text{for all }j,k\in\{1,\dots,m\}.
    \end{equation} 
    We observe that $\{\Pi_a\varphi_i\}_{i=1,\ldots,m}$ is a basis of
    $F_a=\Pi_a(E(\lambda_N^0))$, but $\Pi_a\varphi_i$ are not
    necessarily orthogonal to each other in $L^2(\Omega)$. By Lemma
    \ref{l:restrict} and direct calculations, it follows that the
    eigenvalues $\{\xi_j^a\}_{j=1}^m$ of $q_a\big|_{F_a}$ coincide
    with the eigenvalues of the matrix
    \begin{equation*}
\mathsf{B}_a=\mathsf{C}_a^{-1}\mathsf{R}_a,
    \end{equation*}
    where $\mathsf{C}_a$ is the $m\times m$ symmetric real matrix 
    \begin{equation*}
\mathsf{C}_a=\Big((\Pi_a\varphi_j,\Pi_a\varphi_k)_{L^2(\Omega)}\Big)_{jk}.
    \end{equation*}
By Lemma \ref{l:prop4.6FNOS} and \eqref{eq:chiEps} we deduce that
\begin{equation*}
    \mathsf{C}_a=\mathbb{I}+o(\chi_a)\quad\text{as }|a|\to0,
\end{equation*}
with $\mathbb{I}$ being the identity $m\times m$ matrix, in the sense
that all the entries of the matrix $\mathsf{C}_a-\mathbb{I}$ are
$o(\chi_a)$ as $|a|\to0$.  Hence
$\mathsf{C}_a^{-1}=\mathbb{I}+o(\chi_a)$ as $|a|\to0$, so that
\eqref{eq:mu-xi-1} yields
\begin{equation}\label{eq:mu-xi-2}
    \mathsf{B}_a=(\mathbb{I}+o(\chi_a))\mathsf{R}_a=\mathsf{R}_a+o(\chi_a^2)\quad\text{as }|a|\to0.
\end{equation}
The expansion \eqref{eq:exp-mu-xi} follows from \eqref{eq:mu-xi-2} and
the min-max characterization of eigenvalues.
\end{proof}

Therefore, in view of \eqref{eq:asympt_proof} and Lemma \ref{l:mu-xi},
the proof of Theorem \ref{thm:approxEVs} is complete.

\section{Study of the function \texorpdfstring{$\mathcal C(\alpha,u)$}{Calphau}} 
\label{sec:functionC}

In the present section, we study the properties of the function
$\mathcal C(\alpha,u)$ defined in \eqref{d:functionC}, which plays a
crucial role in the asymptotic expansion of eigenbranches, as emerged
in Section \ref{sec:appLemma}.

From \eqref{eq:psi} it follows that, for every
$u\in E(\lambda_N^0)\setminus\{0\}$ and $\alpha\in(-\pi,\pi]$,
    \begin{equation*}
\Psi_\alpha^u(r\cos t,r\sin t)= \beta(u) f_\alpha(t)
F_u(r\cos t,r\sin t)
\end{equation*}
where $\beta(u)\neq0$ and 
 \begin{equation*}
 F_u(r\cos t,r\sin t)=r^{\frac {k(u)}2}\sin\Big(\tfrac{k(u)}{2} (t-\omega(u))\Big).
\end{equation*} 
Furthermore, the restrictions of $\Psi^u_\alpha$ and
$\nabla \Psi^u_\alpha\cdot \nu_\alpha$ on ${S^\alpha_1}$ are given
respectively by
\begin{equation*}
   \Psi^u_\alpha(r\cos\alpha,r\sin\alpha)=
   \beta(u) f_\alpha(\alpha)
   \sin\left(\tfrac{k(u)}{2} (\alpha-\omega(u))\right)r^{\frac {k(u)}2}\quad\text{for all }r\in [0,1],
\end{equation*}
and
\begin{equation*}
  \nabla \Psi^u_\alpha(r\cos\alpha,r\sin\alpha)\cdot \nu_\alpha=
  \tfrac {k(u)}2\, \beta(u) f_\alpha(\alpha)
  \cos\left(\tfrac{k(u)}{2} (\alpha-\omega(u))\right)r^{\frac {k(u)}2-1}\quad\text{for all }r\in [0,1].
\end{equation*}
Therefore, 
from \eqref{d:functionC},  \eqref{d:Jtilde}, \eqref{d:Ltilde}, 
and \eqref{d:Etilde}, for every 
$u\in E(\lambda_N^0)\setminus\{0\}$ and $\alpha\in (-\pi,\pi]$ we have that
\begin{align*}
    &\mathcal C(\alpha,u)\\
    &=2
     \min\left\{ 
     \begin{array}{rr}
    &\hskip-15pt\dfrac12{\displaystyle{\int_{\R^2\setminus\Gamma^\alpha_1}\!\!|\nabla w|^2dx}}\!+\!
     k(u)\beta(u) f_\alpha(\alpha)
   \cos\left(\tfrac{k(u)}{2} (\alpha-\omega(u))\right)
   \!\!{\displaystyle{\int_{S^\alpha_1}\!\!
      |x|^{\frac {k(u)}2-1}
     \gamma^\alpha_+(w-\Psi_\alpha^u)dS}}:\\[10pt]
     &w\in \widetilde X_\alpha\text{ and } w- \eta \Psi_\alpha^u \in\widetilde{\mathcal H}_\alpha
     \end{array}\hskip-5pt\right\}
     \\[10pt]
     &=2
     \min\left\{ \hskip-10pt
     \begin{array}{rr}
       &\hskip-5pt\dfrac12{\displaystyle{\int_{\R^2\setminus\Gamma^\alpha_1}|\nabla w|^2\,dx}}+
         k(u)\beta(u) f_\alpha(\alpha)
         \cos\left(\tfrac{k(u)}{2} (\alpha-\omega(u))\right)
         {\displaystyle{\int_{S^\alpha_1}
         |x|^{\frac {k(u)}2-1}
         \gamma^\alpha_+(w)\,dS}}
       \\[10pt]
       &-(\beta(u))^2 \cos\left(\tfrac{k(u)}{2} (\alpha-\omega(u))\right)\sin\left(\tfrac{k(u)}{2} (\alpha-\omega(u))\right):\\[10pt]
       &w\in \widetilde X_\alpha\text{ and } 
         \gamma^\alpha_+(w)+\gamma^\alpha_-(w)=2\beta(u)f_\alpha(\alpha)
         \sin\left(\tfrac{k(u)}{2} (\alpha-\omega(u))\right)|x|^{\frac {k(u)}2}\text{ on }S^\alpha_1
     \end{array}
     \right\}.
\end{align*}
Replacing $w$ with $\beta(u)f_\alpha(\alpha)v$ in the functional to be minimized, we obtain 
\begin{equation*}
    \mathcal C(\alpha,u)=2(\beta(u))^2
     \min\left\{ \hskip-5pt
     \begin{array}{rr}
       &\hskip-10pt\dfrac12{\displaystyle{\int_{\R^2\setminus\Gamma^\alpha_1}|\nabla v|^2\,dx}}+
         k(u)   \cos\left(\tfrac{k(u)}{2} (\alpha-\omega(u))\right)
         {\displaystyle{\int_{S^\alpha_1}
         |x|^{\frac {k(u)}2-1}
         \gamma^\alpha_+(v)\,dS}}
       \\[10pt]
       &- \cos\left(\tfrac{k(u)}{2} (\alpha-\omega(u))\right)\sin\left(\tfrac{k(u)}{2} (\alpha-\omega(u))\right):\\[10pt]
       &v\in \widetilde X_\alpha\text{ and } 
         \gamma^\alpha_+(v)+\gamma^\alpha_-(v)=2
         \sin\left(\tfrac{k(u)}{2} (\alpha-\omega(u))\right)|x|^{\frac {k(u)}2}\text{ on }S^\alpha_1
     \end{array}
     \right\}.
\end{equation*}
By rotation, we can then rewrite $\mathcal C(\alpha,u)$ as 
\begin{align}\label{eq:cara-C}
    &\mathcal C(\alpha,u)\\
    \notag&=2(\beta(u))^2
     \min\left\{ \hskip-5pt
     \begin{array}{rr}
       &\hskip-10pt\dfrac12{\displaystyle{\int_{\R^2\setminus\Gamma^0_1}|\nabla v|^2\,dx}}+
         k(u)   \cos\left(\tfrac{k(u)}{2} (\alpha-\omega(u))\right)
         {\displaystyle{\int_{S^0_1}
         |x|^{\frac {k(u)}2-1}
         \gamma^0_+(v)\,dS}}
       \\[10pt]
       &- \cos\left(\tfrac{k(u)}{2} (\alpha-\omega(u))\right)\sin\left(\tfrac{k(u)}{2} (\alpha-\omega(u))\right):\\[10pt]
       &v\in \widetilde X_0\text{ and } 
         \gamma^0_+(v)+\gamma^0_-(v)=2
         \sin\left(\tfrac{k(u)}{2} (\alpha-\omega(u))\right)|x|^{\frac {k(u)}2}\text{ on }S^0_1
     \end{array}
     \right\}\\[10pt]
\notag     &=2(\beta(u))^2 G_{k(u)}\left(\tfrac{k(u)}{2} (\alpha-\omega(u))\right),
\end{align}
where, for every $\zeta\in\R$ and $k\in\N$ odd,
\begin{equation}\label{eq:defG}
    G_k(\zeta)=\min\left\{ \hskip-5pt
     \begin{array}{rr}
    &\hskip-10pt\dfrac12{\displaystyle{\int_{\R^2\setminus\Gamma^0_1}|\nabla v|^2\,dx}}+
     k   \cos\zeta
   {\displaystyle{\int_{S^0_1}
      |x|^{\frac k2-1}
     \gamma^0_+(v)\,dS}}- \cos\zeta\sin\zeta:\\[10pt]
     &v\in \widetilde X_0\text{ and } 
     \gamma^0_+(v)+\gamma^0_-(v)=2
     |x|^{\frac k2}\sin\zeta \text{ on }S^0_1
     \end{array}
     \right\}.
\end{equation}
Some properties of the function $G$ are described in the following lemma.
\begin{lemma}\label{l:prop-G}
    For some fixed $k\in\N$ odd, let $G=G_k$ be the function defined in \eqref{eq:defG}. Then
    \begin{enumerate}[\rm (i)]
        \item $G\in C^0(\R)$;
        \item $G(0)<0$ and $G(\frac\pi2)>0$;
        \item $G(\zeta+\pi)=G(\zeta)$ for every $\zeta\in\R$;
        \item $G(\pi-\zeta)=G(\zeta)$ for every $\zeta\in\R$;
        \item $G$ is strictly increasing in $(\frac\pi4,\frac\pi2)$;
        \item $G$ is strictly increasing in every interval $I\subset (0,\frac\pi2)$ such that $G> 0$ in $I$. 
    \end{enumerate}
\end{lemma}
\begin{proof}
For the proof of (i) we refer to \cite[Theorem 6.8]{FelliNorisOgnibeneSiclari2023}.

To prove (ii), we observe that
$G(0)=\min_{\widetilde{\mathcal H}_0}J$, where \begin{equation*}
  J(\varphi)=\frac12\int_{\R^2\setminus\Gamma^0_1}|\nabla
  \varphi|^2\,dx+k\int_{S^0_1}|x|^{\frac k2-1}\gamma^0_+(\varphi)\,dS.
\end{equation*}
Fixing some $\varphi\in \widetilde{\mathcal H}_0$ such that
$\int_{S^0_1}|x|^{\frac k2-1}\gamma^0_+(\varphi)\,dS<0$, we observe
that
\begin{equation*}
  J(t\varphi)=
  \frac12t^2\int_{\R^2\setminus\Gamma^0_1}|\nabla \varphi|^2\,dx+kt\int_{S^0_1}|x|^{\frac k2-1}\gamma^0_+(\varphi)\,dS<0
\end{equation*}
for sufficiently small $t>0$, so that
$\min_{\widetilde{\mathcal H}_0}J<0$. On the other hand
\begin{equation*}
    G\bigg(\frac\pi2\bigg)=\min\left\{ \dfrac12\int_{\R^2\setminus\Gamma^0_1}|\nabla v|^2\,dx:
    v\in \widetilde X_0\text{ and } 
     \gamma^0_+(v)+\gamma^0_-(v)=2
     |x|^{\frac k2} \text{ on }S^0_1
     \right\}>0.
\end{equation*}
Replacing $v$ with $-v$ in the functional to be minimized in
\eqref{eq:defG}, we can equivalently characterize $G$ as
\begin{equation*}
    G(\zeta)=\min\left\{ \hskip-5pt
     \begin{array}{rr}
    &\hskip-10pt\dfrac12{\displaystyle{\int_{\R^2\setminus\Gamma^0_1}|\nabla v|^2\,dx}}-
     k   \cos\zeta
   {\displaystyle{\int_{S^0_1}
      |x|^{\frac k2-1}
     \gamma^0_+(v)\,dS}}- \cos\zeta\sin\zeta:\\[10pt]
     &v\in \widetilde X_0\text{ and } 
     \gamma^0_+(v)+\gamma^0_-(v)=-2
     |x|^{\frac k2}\sin\zeta \text{ on }S^0_1
     \end{array}
     \right\}.
\end{equation*}
From this, we directly conclude that $G(\zeta+\pi)=G(\zeta)$ for all
$\zeta\in\R$, thus proving statement (iii).

To prove (iv), we first rewrite \eqref{eq:defG} equivalently as
\begin{align}\label{eq:varieG}
    G(\zeta)&=\min\left\{ \hskip-5pt
     \begin{array}{rr}
    &\hskip-10pt\dfrac12{\displaystyle{\int_{\R^2\setminus\Gamma^0_1}|\nabla v|^2\,dx}}+
     k   \cos\zeta
   {\displaystyle{\int_{S^0_1}
      |x|^{\frac k2-1}
\Big(2
     |x|^{\frac k2}\sin\zeta-      
     \gamma^0_-(v)\Big)\,dS}}- \cos\zeta\sin\zeta:\\[10pt]
     &v\in \widetilde X_0\text{ and } 
     \gamma^0_+(v)+\gamma^0_-(v)=2
     |x|^{\frac k2}\sin\zeta \text{ on }S^0_1
     \end{array}\hskip-4pt
     \right\}\\
\notag     &=\min\left\{ \hskip-5pt
     \begin{array}{rr}
    &\hskip-10pt\dfrac12{\displaystyle{\int_{\R^2\setminus\Gamma^0_1}|\nabla v|^2\,dx}}-
     k   \cos\zeta
   {\displaystyle{\int_{S^0_1}
      |x|^{\frac k2-1}
     \gamma^0_-(v)\,dS}}+ \cos\zeta\sin\zeta:\\[10pt]
     &v\in \widetilde X_0\text{ and } 
     \gamma^0_+(v)+\gamma^0_-(v)=2
     |x|^{\frac k2}\sin\zeta \text{ on }S^0_1
     \end{array}
     \right\}\\
   \notag  &=\min\left\{ \hskip-5pt
     \begin{array}{rr}
    &\hskip-10pt\dfrac12{\displaystyle{\int_{\R^2\setminus\Gamma^0_1}|\nabla v|^2\,dx}}-
     k   \cos\zeta
   {\displaystyle{\int_{S^0_1}
      |x|^{\frac k2-1}
     \gamma^0_+(v)\,dS}}+ \cos\zeta\sin\zeta:\\[10pt]
     &v\in \widetilde X_0\text{ and } 
     \gamma^0_+(v)+\gamma^0_-(v)=2
     |x|^{\frac k2}\sin\zeta \text{ on }S^0_1
     \end{array}
     \right\},
\end{align}
where, in the last identity, we replaced $v$ with its reflection
through $\Sigma_0$ in the functional to be minimized. From the
characterization of $G$ given above, it follows directly that
$G(\pi-\zeta)=G(\zeta)$ for every $\zeta\in\R$, thus proving (iv).

In order to prove  (v), we observe that, in view of \eqref{eq:varieG},
\begin{equation*}
    G(\zeta)=\min\left\{ \hskip-5pt
     \begin{array}{rr}
    &\hskip-10pt\dfrac12{\displaystyle{\int_{\R^2\setminus\Gamma^0_1}|\nabla v|^2\,dx}}-
     k   \cos\zeta
   {\displaystyle{\int_{S^0_1}
      |x|^{\frac k2-1}
     \Big(\gamma^0_+(v)-|x|^{\frac k2}\sin\zeta\Big)\,dS}}:\\[10pt]
     &v\in \widetilde X_0\text{ and } 
     \gamma^0_+(v)+\gamma^0_-(v)=2
     |x|^{\frac k2}\sin\zeta \text{ on }S^0_1
     \end{array}
     \right\}.
\end{equation*}
For every $\zeta\in(0,\frac\pi2)$, by replacing $v$ with $v\sin\zeta$, we obtain that 
\begin{equation}\label{eq:altraG}
    G(\zeta)=\min\left\{ \hskip-5pt
     \begin{array}{rr}
    &\hskip-10pt\dfrac{\sin^2\zeta}2{\displaystyle{\int_{\R^2\setminus\Gamma^0_1}|\nabla v|^2\,dx}}-
     k   \cos\zeta\sin\zeta
   {\displaystyle{\int_{S^0_1}
      |x|^{\frac k2-1}
     \Big(\gamma^0_+(v)-|x|^{\frac k2}\Big)\,dS}}:\\[10pt]
     &v\in \widetilde X_0\text{ and } 
     \gamma^0_+(v)+\gamma^0_-(v)=2
     |x|^{\frac k2} \text{ on }S^0_1
     \end{array}
     \right\}.
\end{equation}
We observe that, if $v\in \widetilde X_0$ and  
     $\gamma^0_+(v)+\gamma^0_-(v)=2
     |x|^{\frac k2}$ on $S^0_1$, then necessarily 
     \begin{equation*}
     \text{either}\quad \int_{S^0_1}
      |x|^{\frac k2-1}
\gamma^0_+(v)\,dS\geq \int_{S^0_1}
      |x|^{k-1}\,dS\quad\text{or}\quad \int_{S^0_1}
      |x|^{\frac k2-1}
\gamma^0_-(v)\,dS\geq \int_{S^0_1}
      |x|^{k-1}\,dS.
      \end{equation*}
      Hence, we may minimize in \eqref{eq:altraG} only among functions
      such that
      $\int_{S^0_1} |x|^{\frac k2}( \gamma^0_+(v)-|x|^{k/2})\,dS\geq
      0$: indeed, for every $v$ for which that integral is negative,
      the functional takes a smaller value on the reflection of $v$
      through $\Sigma_0$ (which, on the other hand, remains in the
      constraint). Therefore
\begin{equation}\label{eq:GminV}
    G(\zeta)=\min_{\mathcal V}\mathcal J_\zeta
     \end{equation}
where
\begin{equation*}
    \mathcal J_\zeta(v)= \dfrac{\sin^2\zeta}2\int_{\R^2\setminus\Gamma^0_1}|\nabla v|^2\,dx-
     k   \cos\zeta\sin\zeta
   \int_{S^0_1}
      |x|^{\frac k2-1}
     \Big(\gamma^0_+(v)-|x|^{\frac k2}\Big)\,dS
\end{equation*}
and
\begin{equation*}
\mathcal V=\left\{
     v\in \widetilde X_0:
     \gamma^0_+(v)+\gamma^0_-(v)=2
     |x|^{\frac k2} \text{ on }S^0_1 \text{ and }
     \int_{S^0_1}
      |x|^{\frac k2}(
\gamma^0_+(v)-|x|^{k/2})\,dS\geq 0\right\}.
\end{equation*}
We observe that, for every $c_1,c_2\geq0$ with $(c_1,c_2)\not=(0,0)$,
the function $\zeta\mapsto c_1\sin^2\zeta-c_2\sin\zeta\cos\zeta$ is
strictly increasing in $(\frac\pi4,\frac\pi2)$. Hence, for every
$v\in\mathcal V$, the function $\zeta\mapsto \mathcal J_\zeta(v)$ is
strictly increasing in $(\frac\pi4,\frac\pi2)$; passing to the minimum
this directly implies that $G$ is strictly increasing in
$(\frac\pi4,\frac\pi2)$, thus proving (v).

We finally prove (vi). We assume that $I\subset (0,\frac\pi2)$ is an
interval such that $G> 0$ in $I$. For every $\zeta\in(0,\frac\pi2)$,
we can rewrite \eqref{eq:GminV} as
\begin{equation*}
    G(\zeta)=(\sin^2\zeta) g(\zeta),
\end{equation*}
where 
\begin{equation*}
    g(\zeta)=\min_{v\in \mathcal V}
    \left\{\frac12\int_{\R^2\setminus\Gamma^0_1}|\nabla v|^2\,dx-
     k  \frac{\cos\zeta}{\sin\zeta}
   \int_{S^0_1}
      |x|^{\frac k2-1}
     \Big(\gamma^0_+(v)-|x|^{\frac k2}\Big)\,dS\right\}.
\end{equation*}
Since the function $\zeta\mapsto \frac{\cos\zeta}{\sin\zeta}$ is
strictly decreasing in $(0,\frac\pi2)$, the function $g$ is non
decreasing in $(0,\frac\pi2)$. This, together with the fact that $g>0$
in $I$ and $\zeta\mapsto\sin^2\zeta$ is strictly increasing in $I$,
implies that $G$ is strictly increasing in $I$, as stated in (vi).
\end{proof}

\begin{corollary}\label{cor:prop-G}
  There exists $\zeta_0\in (0,\frac\pi2)$ such that $G(\zeta)>0$ for
  every $\zeta\in (\zeta_0,\frac\pi2]$ and $G(\zeta)\leq 0$ for every
  $\zeta\in [0,\zeta_0]$. Furthermore $G(\frac\pi2)>G(\zeta)$ for
  every $\zeta\in [0,\pi]\setminus\{\frac\pi2\}$.
\end{corollary}
\begin{proof}
  The existence of such a $\zeta_0$ follows by combining (i), (ii),
  and (vi) of Lemma \ref{l:prop-G} and letting
  $\zeta_0=\inf\{\zeta\in (0,\frac\pi2): G(\zeta)>0\}$. This, together
  with (v) and (vi) of Lemma \ref{l:prop-G}, directly implies that
  $\zeta=\frac\pi2$ is a strict global maximum point for $G$, which is
  the only one over the period $[0,\pi]$.
\end{proof}

\begin{remark}\label{rem:periodicity}
  We observe that \eqref{eq:cara-C} and Lemma \ref{l:prop-G}-(iii)
  imply that, for every $u\in E(\lambda_N^0)\setminus\{0\}$,
    \begin{equation*}
        \text{the function}\quad \alpha\mapsto \mathcal C(\alpha,u)\quad\text{is $\tfrac{2\pi}k$-periodic}.
    \end{equation*}
\end{remark}
\begin{theorem}\label{t:trasv}\quad 

  \begin{enumerate}[\rm (i)]
  \item If $u\in E(\lambda_N^0)\setminus\{0\}$, then there exist two  intervals $I_{u}^1,I_{u}^2\subset \R$ 
    with lengths $0<|I_{u}^1|<\frac\pi{k(u)}$ and $0<|I_{u}^2|<\frac\pi{k(u)}$ such that 
    \begin{equation*}
        \mathcal C(\alpha,u)< 0\quad\text{for all }\alpha    \in I_{u}^1+\tfrac{2\pi}{k(u)}\Z\quad\text{and}\quad 
        \mathcal C(\alpha,u)> 0\quad\text{for all }\alpha    \in I_{u}^2+\tfrac{2\pi}{k(u)}\Z.
    \end{equation*}
  \item If $u,v\in E(\lambda_N^0)\setminus\{0\}$ are such that
    $k(u)=k(v)=k$ and $\omega(u)\neq\omega(v)$, then there exists an
    interval $I_{u,v}\subset \R$ with positive length
    $0<|I_{u,v}|<\frac{2\pi}{k}$ such that
    \begin{equation*}
      \mathcal C(\alpha,u)\neq  \mathcal C(\alpha,v)\quad\text{for all }\alpha    \in I_{u,v}+\tfrac{2\pi}k\Z.
    \end{equation*}
\end{enumerate}
\end{theorem}
\begin{proof}
Let $u\in E(\lambda_N^0)\setminus\{0\}$.  By \eqref{eq:cara-C} and Lemma \ref{l:prop-G}-(ii) we have 
    \begin{equation*}
        \mathcal C(\omega(u), u)=2(\beta(u))^2 
        G_{k(u)}(0)<0\quad \text{and}\quad        
        \mathcal C\left(\omega(u)+\tfrac\pi{k(u)}, u\right)=2(\beta(u))^2 
        G_{k(u)}\left(\tfrac\pi2\right)>0,
    \end{equation*}
    so that statement (i) follows from the fact that the map
    $\alpha\mapsto \mathcal C(\alpha,u)$ is continuous in view of
    \eqref{eq:cara-C} and Lemma \ref{l:prop-G}-(i) and
    $\frac{2\pi}k$-periodic as observed in Remark
    \ref{rem:periodicity}.

    To prove (ii), it is not restrictive to assume
    $|\beta(v)|\leq|\beta(u)|$.  We are going to prove that, letting
    $\alpha_0=\omega(u)+\frac\pi{k(u)}$, we have
    \begin{equation}\label{eq:stat-transv}
        \mathcal C(\alpha_0,u)\neq         \mathcal C(\alpha_0,v).
    \end{equation}
    Once \eqref{eq:stat-transv} is proved, the conclusion follows from
    the continuity and the $\frac{2\pi}k$-periodicity of the map
    $\alpha\mapsto \mathcal C(\alpha,u)$.  To prove
    \eqref{eq:stat-transv} we observe that \eqref{eq:cara-C} implies
    \begin{equation*}
        \mathcal C(\alpha_0, u)=2(\beta(u))^2 
        G_{k}\left(\tfrac\pi2\right)\quad\text{and}\quad
        \mathcal C(\alpha_0, v)=2(\beta(v))^2 
        G_{k}\left(\tfrac\pi2+\omega_0\right),
    \end{equation*}
    where $k=k(u)=k(v)$ and
    $\omega_0=\tfrac{k}{2} (\omega(u)-\omega(v))$. Since
    $0<|\omega_0|<\pi$, from Corollary \ref{cor:prop-G} and Lemma
    \ref{l:prop-G}-(iii) it follows that
    $G_{k}\left(\tfrac\pi2+\omega_0\right)<G_{k}\left(\tfrac\pi2\right)$. Since
    $(\beta(v))^2 \leq (\beta(u))^2$ we conclude that
    $\mathcal C(\alpha_0,v)< \mathcal C(\alpha_0,u)$, thus proving
    \eqref{eq:stat-transv}.
\end{proof}

\section{Decomposition of the eigenspace}\label{sec:decomposition}

We provide here a suitable decomposition of the eigenspace
$E(\lambda_N^0)$ that will be useful to detect the behavior of
different perturbed eigenvalues, when the operator's pole is moving to
$0$ along a fixed direction.  See \cite[Proposition 1.10]{ALM2022} for
its counterpart for the Dirichlet Laplacian.

\begin{proposition}\label{prop:DecompES} 
  Under assumptions \eqref{eq:multiple}--\eqref{eq:index}, let
  $E(\lambda_N^0)$ be the eigenspace introduced in
  \eqref{eq:Elambda0}.  There exist a decomposition of
  $E(\lambda_N^0)$ into a sum of orthogonal subspaces
\[
E(\lambda_N^0)=E_1\oplus\dots\oplus E_p,
\] 
for some integer $p\geq 1$, with $\dim E_\ell=m_\ell\in\N\setminus\{0\}$, and 
$p$ distinct odd  natural numbers 
\[
k_1<\dots<k_p \in \N
\]
such that $k(\varphi)=k_\ell$ for every
$\varphi\in E_\ell\setminus\{0\}$ and $1\le \ell \le p$, i.e.,
according to the notation introduced in Proposition \ref{p:asyeige},
for every $\varphi\in E_\ell\setminus\{0\}$ there exist
$\beta(\varphi)\in\R\setminus\{0\}$ and
$\omega(\varphi)\in [0,\tfrac{2\pi}{k_\ell})$ such that
\begin{equation}\label{eq:orderofvanishing}
  r^{-k_\ell/2}\varphi(r\cos t,r\sin t)\to \beta(\varphi)
  e^{i\tfrac{t}{2}}
  \sin\Big(\tfrac{k_\ell}{2}(t-\omega(\varphi)\Big) \quad\text{in }C^{1,\tau}([0,2\pi],\C]),~\text{as }r\to 0.
\end{equation}
Moreover, if, for some $\ell\in\{1,\ldots,p\}$,
$\varphi,\psi \in E_\ell$ are linearly independent, then
$\omega(\varphi)\neq \omega(\psi)$.
\end{proposition}
\begin{proof}
    For every $j\in \N$, let $N_j$ be the linear subspace of $E(\lambda_N^0)$ given by 
    \[
    N_j=\{ u\in E(\lambda_N^0): k(u)\ge 2(j+1)+1\},
    \]
    with the agreement that $k(u)=+\infty$ if $u\equiv0$.  Then
    $\bigcap_{j\in\N}N_j=\{0\}$ and there exist $j_1\in\N$ such that
    $N_j=\{0\}$ for any $j\ge j_1$ and $N_j\neq \{0\}$ for any
    $j<j_1$.  Let
    \[
    J:=\{j\in\N: \dim N_{j-1}>\dim N_{j}  \},
    \]
    under the convention that $N_{-1}=E(\lambda_N^0)$.  Then
    $J\neq \emptyset$ because $j_1\in J$; moreover, $J$ is bounded
    because, if $j\in J$, then $j\le j_1$. Thus,
    \[
    J=\{j_p,\ldots,j_1\} \quad \text{with }0\le j_p<j_{p-1}<\ldots<j_1
    \]
    and 
    \[
    N_{j_1}=\{0\} \subset N_{j_2} \subset \ldots \subset N_{j_p} \subset E(\lambda_N^0)
    \]
    with strict inclusion. 
    Let us now define, for every $\ell\in\{1,\ldots,p\}$, the space
    \[
    E_\ell := N_{j_{\ell+1}} \cap N_{j_\ell}^\perp
    \]
    with the agreement that $j_{p+1}=-1$, so that $N_{j_{p+1}}=N_{-1}=E(\lambda_N^0)$. In this way 
    \[
    E(\lambda_N^0)=E_1\oplus\dots\oplus E_p,
    \]
    and, if $u\in E_\ell$, then its vanishing order is
    $k(u)\ge 2j_\ell+1$. If $k(u)> 2j_\ell+1$, then
    $k(u)\geq 2(j_\ell+1)+1$, so that $u\in N_{j_\ell}$; then, since
    $u\in N_{j_\ell}^\perp$, we necessarily have $u=0$. Therefore it
    must be $k(u)=2j_\ell+1$ for every $u\in E_\ell\setminus\{0\}$,
    proving the first part of the proposition with $k_\ell=2j_\ell+1$.

    Let us now consider $\varphi$ and $\psi$ in $E_\ell\setminus\{0\}$
    for some $\ell\in\{1,\dots,p\}$. We are going to prove that, if
    $\omega(\varphi)=\omega(\psi)=\omega_0$, then $\varphi$ and $\psi$
    are linearly dependent.  From \eqref{eq:orderofvanishing} we
    deduce
    \[
      r^{-\frac{k_\ell}{2}}\big( \beta(\varphi) \psi(r\cos t,r\sin t )
      - \beta(\psi)\varphi(r\cos t,r\sin t )\big) \to 0 \quad \text{as
      }r\to0
    \]
    uniformly with respect to $t$, so that
    $\beta(\varphi) \psi - \beta(\psi)\varphi \in E_\ell$ is such that
    $k\big(\beta(\varphi) \psi - \beta(\psi)\varphi\big)>k_\ell$.
    This implies that
    $\beta(\varphi) \psi - \beta(\psi)\varphi \equiv 0$, that is
    $\varphi$ and $\psi$ are linearly dependent. This concludes the
    proof.
    \end{proof}

    In the case where the eigenvalue $\lambda_N^0$ is double, i.e., if
    $\dim E(\lambda_N^0)=2$, the above decomposition result leads to
    the two alternative situations described in the following
    corollary.
\begin{corollary}\label{cor:m=2}
  If $m=\dim E(\lambda_N^0)=2$, then one of the following alternatives
  must hold:
\begin{enumerate}[\rm (i)]
\item there exist $k\in\N$ odd and an orthonormal basis
  $\{\varphi,\psi\}$ of $E(\lambda_N^0)$ such that
  $k(\varphi)=k(\psi)=k$ and $\omega(\varphi)\neq\omega(\psi)$;
\item there exist $k_1,k_2\in\N$ odd and an orthonormal basis
  $\{\varphi,\psi\}$ of $E(\lambda_N^0)$ such that $k(\varphi)=k_1$,
  $k(\psi)=k_2$, and $k_1<k_2$.
    \end{enumerate}
\end{corollary}
\begin{proof}
  Let $p$ be as in Proposition \ref{prop:DecompES}. Since $m=2$, we
  have either $p=1$ or $p=2$.

    If $p=1$ alternative (i) occurs, with $k=k_1$ and
    $\{\varphi,\psi\}$ being any orthonormal basis of the space
    $E_1=E(\lambda_N^0)$.

    If $p=2$ alternative (ii) occurs. Indeed, in this case Proposition
    \ref{prop:DecompES} provides an orthogonal decomposition
    \begin{equation*}
    E(\lambda_N^0)=E_1\oplus E_2,
    \end{equation*}
    with $\dim E_1=\dim E_2=1$, and odd natural numbers $k_1<k_2$ such
    that $k(\varphi)=k_1$ for every $\varphi\in E_1\setminus\{0\}$ and
    $k(\psi)=k_2$ for every $\psi\in E_2\setminus\{0\}$. Then (ii)
    follows choosing $\varphi\in E_1$ with
    $\int_\Omega |\varphi|^2\,dx=1$ and $\psi\in E_2$ with
    $\int_\Omega |\psi|^2\,dx=1$.
\end{proof}

\section{Ramification of eigenbranches from double eigenvalues}
\label{sec:ramification}

The study of the function $\mathcal C(\alpha,u)$, performed in Section
\ref{sec:functionC}, allows us to identify some directions for which
the eigenbranches in \eqref{eq:asymptEV} have different leading terms,
thus bifurcating. We succeed in doing this only in the case of double
eigenvalues, thus deducing the main result of the paper. The case of
eigenvalues with multiplicity greater than $2$ presents significantly
greater difficulties, due to intertwining dependence between the fixed
basis and the resulting quadratic form \eqref{eq:def-r_a}.  This will
be the object of a future investigation.

For $\alpha\in(-\pi,\pi]$ and $a=|a|(\cos\alpha,\sin\alpha)$, let
$r_a$ be the bilinear form defined in \eqref{eq:def-r_a}.  In view of
Proposition \ref{p:asyeige}, Theorem \ref{t:teo2.2FNOS}, and Lemma
\ref{l:prop4.6FNOS} we have
\begin{equation}\label{eq:ra1}
    r_a(u,w)=|a|^{\frac{k(u)+k(w)}{2}}\left(\mathcal R_\alpha (u,w)+o(1)\right)
  \quad\text{as }|a|\to0  \end{equation}
for every $u,w\in E(\lambda_N^0)\setminus\{0\}$, where 
\begin{align*}
    \mathcal R_\alpha (u,w)=&
    2\int_{S_1^\alpha}(\nabla \Psi^u_\alpha\cdot\nu_\alpha)  \gamma^\alpha_+(\widetilde{U}_\alpha^w)\,dS-2
    \int_{S_1^\alpha}(\nabla \Psi^u_\alpha\cdot\nu_\alpha) \Psi^w_\alpha\,dS\\
&\quad     +2\int_{S_1^\alpha}(\nabla \Psi^w_\alpha\cdot\nu_\alpha)  \gamma^\alpha_+(\widetilde{U}_\alpha^u)\,dS-2
    \int_{S_1^\alpha}(\nabla \Psi^w_\alpha\cdot\nu_\alpha) \Psi^u_\alpha\,dS
    +\int_{\R^2\setminus \Gamma_1^\alpha} \nabla \widetilde{U}_\alpha^u\cdot \nabla \widetilde{U}_\alpha^w\,dx.
\end{align*}
Furthermore, by \eqref{eq:rauu}   
\begin{equation}\label{eq:raca}
   \mathcal R_\alpha (u,u) =\mathcal C(\alpha,u)\quad\text{for every }u\in 
   E(\lambda_N^0)\setminus\{0\},
\end{equation}
where $\mathcal C(\alpha,u)$ is defined in \eqref{d:functionC}.

Let us assume that $m=\dim E(\lambda_N^0)=2$. Then, either alternative
(i) or alternative (ii) of Corollary \ref{cor:m=2} are
satisfied. Let's analyze these two cases one by one.

\smallskip {\bf Case 1: alternative (i) of Corollary \ref{cor:m=2} is
  satisfied}. Let $k\in\N$ odd and $\{\varphi,\psi\}$ be an
orthonormal basis of $E(\lambda_N^0)$ such that $k(\varphi)=k(\psi)=k$
and $\omega(\varphi)\neq\omega(\psi)$.  Then the eigenvalues
$\{\mu_j^a\}_{j=1}^2$ appearing in \eqref{eq:exp-mu-xi} coincide with
the eigenvalues of the $2\times 2$ symmetric real matrix
\begin{equation}\label{eq:matrixRa-m=2}
        \mathsf{R}_a=
        \begin{pmatrix}
          r_a(\varphi,\varphi)&  r_a(\varphi,\psi)\\
          r_a(\psi,\varphi)&r_a(\psi,\psi)
        \end{pmatrix}.
            \end{equation}
In view of \eqref{eq:ra1} and \eqref{eq:raca} the matrix $\mathsf{R}_a$ satisfies 
\begin{equation*}
        \mathsf{R}_a=
        |a|^k\left(
        \mathsf{R}(\alpha,\varphi,\psi)+o(1)\right),   
            \end{equation*}
 where 
\begin{equation}\label{eq:matrixRa}
            \mathsf{R}(\alpha,\varphi,\psi)=
        \begin{pmatrix}
          \mathcal C(\alpha,\varphi)& \mathcal R_\alpha(\varphi,\psi)\\
          \mathcal R_\alpha(\psi,\varphi)& \mathcal C(\alpha,\psi)
        \end{pmatrix}
\end{equation}
and $o(1)$ stands for a matrix with all entries vanishing to $0$ as
$|a|\to0$.  Denoting as $\mu_1(\alpha,\varphi,\psi)$ and
$\mu_2(\alpha,\varphi,\psi)$ the two real eigenvalues of the symmetric
real matrix $\mathsf{R}(\alpha,\varphi,\psi)$ (in ascending order),
the min-max characterization of eigenvalues ensures that
\begin{equation*}
  \mu_j^a=|a|^k\big(\mu_j(\alpha,\varphi,\psi)+o(1)\big)\quad\text{as }|a|\to0,\quad j=1,2.
\end{equation*}
Furthermore, from Theorem \ref{t:teo2.2FNOS} and Proposition
\ref{p:asyeige} the quantity $\chi_a$ defined in \eqref{eq:chiEps}
satisfies
\begin{equation*}
    \chi_a=O(|a|^{k/2})\quad\text{as }|a|\to0.
\end{equation*}
 From  Theorem \ref{thm:approxEVs} it then follows that 
\begin{equation}\label{eq:alt(i)}
    \lambda_{N+j-1}^a=\lambda_N^0+|a|^k\mu_j(\alpha,\varphi,\psi)+o(|a|^k) \mbox{ as }|a|\to0,\quad j=1,2.
\end{equation}
\smallskip {\bf Case 2: alternative (ii) of Corollary \ref{cor:m=2} is
  satisfied}.  Let $k_1,k_2\in\N$ odd and $\{\varphi,\psi\}$ be an
orthonormal basis of $E(\lambda_N^0)$ such that $k(\varphi)=k_1$,
$k(\psi)=k_2$, with $k_1<k_2$. From \eqref{eq:ra1} and \eqref{eq:raca}
it follows that
\begin{align*}
   & r_a(\varphi,\varphi)=|a|^{k_1}\mathcal C_\alpha(\alpha,\varphi)+o(|a|^{k_1})
  \quad\text{as }|a|\to0,\\
   & r_a(\varphi,\psi)=r_a(\psi,\varphi)=
   |a|^{\frac{k_1+k_2}{2}}\left(\mathcal R_\alpha (\varphi,\psi)+o(1)\right)=o(|a|^{k_1})
  \quad\text{as }|a|\to0, \\
  & r_a(\psi,\psi)=|a|^{k_2}\mathcal C_\alpha(\alpha,\psi)+o(|a|^{k_2})=o(|a|^{k_1})
  \quad\text{as }|a|\to0.
\end{align*}
Hence, in this case, the matrix $\mathsf{R}_a$ in \eqref{eq:matrixRa-m=2} satisfies
\begin{equation*}
    \mathsf{R}_a=|a|^{k_1}\left(
    \begin{pmatrix}
          \mathcal C(\alpha,\varphi)& 0\\
          0& 0
        \end{pmatrix}+o(1)\right)
        \quad\text{as }|a|\to0.
\end{equation*}
The min-max characterization of eigenvalues ensures that either
\begin{equation*}
\mu_1^a=|a|^{k_1}\mathcal C(\alpha,\varphi)+o(|a|^{k_1}),\quad 
\mu_2^a=o(|a|^{k_1})\text{ as }|a|\to0, \quad \text{if $\mathcal C(\alpha,\varphi)\leq 0$},
\end{equation*}
or 
\begin{equation*}
  \mu_1^a=o(|a|^{k_1}),\quad 
  \mu_2^a=|a|^{k_1}\mathcal C(\alpha,\varphi)+o(|a|^{k_1})
  \text{as }|a|\to0,\quad\text{if $\mathcal C(\alpha,\varphi)\geq 0$}.
\end{equation*}
By Theorem \ref{t:teo2.2FNOS} and Proposition \ref{p:asyeige} we have  
\begin{equation*}
    \chi_a=O(|a|^{k_1/2})\quad\text{as }|a|\to0.
\end{equation*}
Therefore, Theorem \ref{thm:approxEVs} yields either
\begin{equation}\label{eq:neg}
\lambda_{N}^a=\lambda_N^0+|a|^{k_1}\mathcal C(\alpha,\varphi)+o(|a|^{k_1})\quad
\text{and}\quad 
\lambda_{N+1}^a=\lambda_N^0+o(|a|^{k_1})\quad\text{as }|a|\to0
\end{equation}
if $\mathcal C(\alpha,\varphi)\leq0$, or 
\begin{equation}\label{eq:pos}
\lambda_{N}^a=\lambda_N^0+o(|a|^{k_1})\quad
\text{and}\quad 
\lambda_{N+1}^a=\lambda_N^0+|a|^{k_1}\mathcal C(\alpha,\varphi)+o(|a|^{k_1})\quad\text{as }|a|\to0
\end{equation}
if $\mathcal C(\alpha,\varphi)\geq0$.

We are now in position to prove Theorem \ref{t:genericity}.
\begin{proof}[Proof of Theorem \ref{t:genericity}]
  Since $m=\dim E(\lambda_N^0)=2$, either alternative (i) or
  alternative (ii) of Corollary \ref{cor:m=2} hold.

  If alternative (i) is satisfied, then by \eqref{eq:alt(i)} there
  exists $k\in\N$ odd such that
\begin{equation}\label{eq:expcasei}
    \lambda_{N}^a=\lambda_N^0+|a|^k\mu_1(\alpha,\varphi,\psi)+o(|a|^k) 
  \quad\text{and}\quad   
    \lambda_{N+1}^a=\lambda_N^0+|a|^k\mu_2(\alpha,\varphi,\psi)+o(|a|^k)
    \end{equation}
    as $a=|a|(\cos\alpha,\sin\alpha)$ and $|a|\to0$, with
    $\mu_1(\alpha,\varphi,\psi)\leq \mu_2(\alpha,\varphi,\psi)$ being
    the eigenvalues of the matrix \eqref{eq:matrixRa}, for
    $\{\varphi,\psi\}$ being an orthonormal basis of $E(\lambda_N^0)$
    as in Corollary \ref{cor:m=2}-(i).  Letting $I_{\varphi,\psi}$ be
    as in Theorem \ref{t:trasv}-(ii), for every
    $\alpha\in I_{\varphi,\psi}+\frac{2\pi}k\Z$ we have
    $\mathcal C(\alpha,\varphi)\neq\mathcal C(\alpha,\psi)$, and hence
    $\mu_1(\alpha,\varphi,\psi)\neq \mu_2(\alpha,\varphi,\psi)$;
    indeed a symmetric $2\times 2$ real matrix has distinct
    eigenvalues if the principal diagonal has distinct
    entries. The conclusion then follows from \eqref{eq:expcasei}.

    If alternative (ii) is satisfied, let $\{\varphi,\psi\}$ being an
    orthonormal basis of $E(\lambda_N^0)$ as in Corollary
    \ref{cor:m=2}-(ii), with $k_1=k(\varphi)<k_2=k(\psi)$. Let
    $I_{\varphi}^1,I_{\varphi}^2$ be as in Theorem
    \ref{t:trasv}-(i). If
    $\alpha\in I_{\varphi}^1+\frac{2\pi}{k_1}\Z$, then
    $\mathcal C(\alpha,\varphi)<0$, whereas, if
    $\alpha\in I_{\varphi}^2+\frac{2\pi}{k_1}\Z$, then
    $\mathcal C(\alpha,\varphi)>0$, so that
    \eqref{eq:neg}--\eqref{eq:pos} yield the conclusion.

Invoking Remark \ref{rem:analiticita-multipli} we achieve the
conclusion for $\alpha\in I_{\varphi,\psi}+\frac{\pi}k\Z$ in case (i)
and  $\alpha\in I_{\varphi}^2+\frac{\pi}{k_1}\Z$ in case (ii).
\end{proof}

\section{Applications to symmetric domains}\label{sec:symmetric}

Our results have some relevant consequences when the domain exhibits
symmetry properties.

\subsection{Invariance under a rotation}
As far as invariance by rotation is concerned, we prove Corollary \ref{c:rotdom}.
\begin{proof}[Proof of Corollary \ref{c:rotdom}]
  Let us denote as $\mathcal R_\ell$ the counterclockwise rotation of
  an angle $\tfrac{2\pi}{\ell}$. Under the assumption
  $\mathcal R_\ell(\Omega)=\Omega$, direct computations imply that
  $\lambda^a_k=\lambda_k^{\mathcal R_\ell a}$ for every $a\in\Omega$
  and $k\in\N\setminus\{0\}$. The conclusion then follows directly
  from Theorem \ref{t:genericity}.
\end{proof}
 
\subsection{Axially symmetric domains}\label{sec:axially-symm-doma}
Another relevant case, in which Theorem \ref{t:genericity} can provide more precise
information, arises when the domain $\Omega$ is axially
symmetric. Without losing generality, we assume that $\Omega$ is
symmetric with respect to the $x_1$-axis, i.e.  $x_2=0$.  In this
  case, we can take advantage of some invariance properties of
    the eigenspaces with respect to the operator $\mathfrak S$
    introduced in \cite{BNHHO2009} (see also \cite[Section
    3.2]{Abatangelo2019} and \cite[Section 3.2]{AFHL}) and recalled
    below.

    Let $\sigma:\R^2\to\R^2$, $\sigma(x_1,x_2)=(x_1,-x_2)$, be the
    reflection through the $x_1$-axis and assume that
    $\sigma(\Omega)=\Omega$. We define the antiunitary antilinear
    operator
\[
\mathfrak S:\ L^2 (D)\to L^2 (D),\quad \mathfrak S u := \bar u \circ \sigma,  
\]
$\bar u$ being the conjugate of $u$.  Then $\mathfrak S$ and
$(i\nabla +A_0)^2$ commute; furthermore, $u\in H^{1 ,0}(\Omega,\C)$
satisfies \eqref{eq:propertyP} if and only if $\mathfrak S u$
satisfies \eqref{eq:propertyP}. This implies that $E(\lambda_N^0)$ is
stable under the action of $\mathfrak S$, i.e.  $u\in E(\lambda_N^0)$
if and only if $\mathfrak Su\in E(\lambda_N^0)$.
  
By Proposition \ref{p:asyeige} and direct computations, we obtain the
following characterization of the quantities $\omega(\varphi)$ and
$k(\varphi)$ for eigenfunctions $\varphi$ which are invariant under
the actions of $\pm\mathfrak S$.
\begin{lemma}\label{l:omega-simm}
  Under the assumption $\Omega=\sigma(\Omega)$, let
  $\varphi \in E(\lambda_N^0)$ and $\omega(\varphi)$, $k(\varphi)$ be
  as in Proposition~\ref{p:asyeige}.
\begin{enumerate}[\rm (i)]
\item If $\varphi=-\mathfrak S \varphi$, then $\omega(\varphi)=0$.
\item If $\varphi=\mathfrak S \varphi$, then $\omega(\varphi)=\frac\pi {k(\varphi)}$.
\item If $\omega(\varphi)\neq \frac\pi {k(\varphi)}$, then
  $k(\varphi-\mathfrak S\varphi)=k(\varphi)$.
\item If $\omega(\varphi)\neq 0$, then $k(\varphi+\mathfrak S\varphi)=k(\varphi)$.
\end{enumerate}
\end{lemma}
Furthermore, if $\Omega$ is axially symmetric, it is possible to
derive the following more precise version of Corollary \ref{cor:m=2}.
\begin{corollary}\label{cor:m=2-simm}
  If $\Omega=\sigma(\Omega)$ and $m=\dim E(\lambda_N^0)=2$, then one
  of the following alternatives must hold:
\begin{enumerate}[\rm (i)]
\item there exist $k\in\N$ odd and an orthonormal basis
  $\{\varphi_a,\varphi_s\}$ of $E(\lambda_N^0)$ such that
  $\omega(\varphi_a)=0$, $\omega(\varphi_s)=\frac\pi k$, and
  $k(\varphi_a)=k(\varphi_s)=k$;
\item there exist $k_1,k_2\in\N$ odd and an orthonormal basis
  $\{\varphi_{1},\varphi_2\}$ of $E(\lambda_N^0)$ such that
  $k(\varphi_{1})=k_1$, $k(\varphi_2)=k_2$, $k_1<k_2$, and either
  $\omega(\varphi_1)=0$ or $\omega(\varphi_1)=\frac\pi {k_1}$.
    \end{enumerate}
\end{corollary}
\begin{proof}
  Let $\{\varphi,\psi\}$ be an orthonormal basis of $E(\lambda_N^0)$
  as in Corollary \ref{cor:m=2}, either in case (i) or in case (ii).
  Let us first consider alternative (i) of Corollary
  \ref{cor:m=2}. Three cases can occur:
\begin{description}
\item[Case 1] $\omega(\psi)=0$. In this case, $\omega(\varphi)\neq0$
  by Corollary \ref{cor:m=2}. Then Lemma \ref{l:omega-simm} (i)--(ii)
  implies that $\psi-\mathfrak S \psi\neq 0$ and
  $\varphi+\mathfrak S \varphi\neq0$. Hence we can define
    \begin{equation*}
        \varphi_a=\frac{\psi-\mathfrak S \psi}{\|\psi-\mathfrak S \psi\|}_{L^2_K(\Omega)},\quad
                \varphi_s=\frac{\varphi+\mathfrak S \varphi}{\|\varphi+\mathfrak S \varphi\|_{L^2_K(\Omega)}},
    \end{equation*}
    noting that $k(\varphi_a)=k(\varphi_s)=k$ in view of Lemma
    \ref{l:omega-simm} (iii)--(iv).
  \item[Case 2] $\omega(\varphi)=\frac\pi k$. In this case,
    $\omega(\psi)\neq\frac \pi k$ by Corollary \ref{cor:m=2}. Then
    Lemma \ref{l:omega-simm} (i)--(ii) implies that
    $\varphi+\mathfrak S \varphi\neq 0$ and
    $\psi-\mathfrak S \psi\neq0$. Hence we can define
    \begin{equation*}
        \varphi_a=\frac{\psi-\mathfrak S \psi}{\|\psi-\mathfrak S \psi\|_{L^2_K(\Omega)}},\quad
                \varphi_s=\frac{\varphi+\mathfrak S \varphi}{\|\varphi+\mathfrak S \varphi\|_{L^2_K(\Omega)}},
    \end{equation*}
    noting that $k(\varphi_a)=k(\varphi_s)=k$ in view of  Lemma \ref{l:omega-simm} (iii)--(iv).
  \item[Case 3] $\omega(\varphi)\neq\frac\pi k$ and
    $\omega(\psi)\neq0$. In this case, Lemma \ref{l:omega-simm}
    (i)--(ii) implies that $\varphi-\mathfrak S \varphi\neq 0$ and
    $\psi+\mathfrak S \psi\neq0$. Hence we can define
    \begin{equation*}
        \varphi_a=\frac{\varphi-\mathfrak S \varphi}{\|
        \varphi-\mathfrak S \varphi\|_{L^2_K(\Omega)}},\quad
                \varphi_s=\frac{\psi+\mathfrak S \psi}{\|\psi+\mathfrak S \psi\|_{L^2_K(\Omega)}},
    \end{equation*}
    noting that $k(\varphi_a)=k(\varphi_s)=k$ in view of Lemma
    \ref{l:omega-simm} (iii)--(iv).

\end{description}
In all the cases described above, we have
$\mathfrak S \varphi_a=-\varphi_a$ and
$\mathfrak S \varphi_s=\varphi_s$, so that $\omega(\varphi_a)=0$ and
$\omega(\varphi_s)=\frac\pi k$ by Lemma \ref{l:omega-simm}.
Furthermore, the fact that $\varphi,\psi$ are orthogonal and satisfy
\eqref{eq:propertyP} implies that the functions $\varphi_a$ and
$\varphi_s$ constructed above are orthogonal, thus proving statement
(i).

Let us now consider alternative (ii) of Corollary \ref{cor:m=2}, i.e.
$k(\varphi)=k_1$, $k(\psi)=k_2'$, for some odd $k_1,k_2'$ with
$k_1<k_2'$. We distinguish two cases:
\begin{description}
\item[Case 1] $\omega(\varphi)\neq0$. In this case, Lemma
  \ref{l:omega-simm} (i) implies that
  $\varphi+\mathfrak S \varphi\neq0$. Hence we can define
    \begin{equation*}
      \varphi_{1}=\frac{\varphi+\mathfrak S \varphi}{\|\varphi+\mathfrak S \varphi\|_{L^2_K(\Omega)}},
      \quad 
      \varphi_2=\begin{cases}
        \frac{\psi-\mathfrak S \psi}{\|\psi-\mathfrak S \psi\|_{L^2_K(\Omega)}},&\text{if }\psi\neq \mathfrak S \psi,\\
        \psi,&\text{if }\psi= \mathfrak S \psi.
        \end{cases}
    \end{equation*}
    We have $k(\varphi_{1})=k(\varphi)=k_1$ in view of Lemma
    \ref{l:omega-simm} (iv), and
    $k_2:=k(\varphi_2)\geq k(\psi)=k_2'>k_1$.  Moreover,
    $\mathfrak S \varphi_{1}=\varphi_{1}$, hence
    $\omega(\varphi_{1})=\frac\pi {k_1}$ by Lemma \ref{l:omega-simm}
    (ii).
  \item[Case 2] $\omega(\varphi)=0$. In this case, Lemma
    \ref{l:omega-simm} (ii) implies that
    $\varphi-\mathfrak S \varphi\neq0$. Hence we can define
    \begin{equation*}
        \varphi_{1}=\frac{\varphi-\mathfrak S \varphi}{\|\varphi-\mathfrak S \varphi\|_{L^2_K(\Omega)}}  ,
        \quad 
        \varphi_2=\begin{cases}
            \frac{\psi+\mathfrak S \psi}{\|\psi+\mathfrak S
              \psi\|_{L^2_K(\Omega)}},
            &\text{if }\psi\neq -\mathfrak S \psi,\\
            \psi,&\text{if }\psi= -\mathfrak S \psi.
        \end{cases}
    \end{equation*}
    noting that $k(\varphi_{1})=k(\varphi)=k_1$ in view of Lemma
    \ref{l:omega-simm} (iii), while
    $k_2:=k(\varphi_2)\geq k(\psi)>k_1$.  Moreover, the fact that
    $\mathfrak S \varphi_{1}=-\varphi_{1}$ and Lemma
    \ref{l:omega-simm} (i) imply that $\omega(\varphi_{1})=0$.
\end{description}
In both the above two cases, direct computations yield that
$\varphi_1$ and $\varphi_2$ are orthogonal in $L^2_K(\Omega)$, thus
completing the proof in case (ii).
\end{proof}

We are now in position to prove Corollary \ref{c:symmdom}.
\begin{proof}[Proof of Corollary \ref{c:symmdom}]
  Let's go over the proof of Theorem \ref{t:genericity}, taking into
  account the improved version of Corollary \ref{cor:m=2} given by
  Corollary \ref{cor:m=2-simm}, which is available under the
  assumption of symmetry with respect to the $x_1$-axis.

  If alternative (i) of Corollary \ref{cor:m=2-simm} holds, then
  \eqref{eq:cara-C} and Lemma \ref{l:prop-G} ensure that
\begin{align*}
  &    \mathcal C(0,\varphi_a)=
    2(\beta(\varphi_a))^2 G_{k}(0)<0,\quad 
    \mathcal C(0,\varphi_s)=
    2(\beta(\varphi_s))^2 G_{k}\left(-\tfrac\pi2\right)>0,\\
  &    \mathcal C(\pi,\varphi_a)=
    2(\beta(\varphi_a))^2 G_{k}\left(\tfrac k2 \pi\right)>0,\quad 
    \mathcal C(\pi,\varphi_s)=
    2(\beta(\varphi_s))^2 G_{k}\left(\tfrac k2 \pi-\tfrac\pi2\right)<0.
\end{align*}
Then, for either $\alpha=0$ or $\alpha=\pi$ (and, by continuity, in
some neighbourhoods of $0$ and $\pi$), the entries on the diagonal of
the matrix \eqref{eq:matrixRa-m=2} are distinct. It follows that, in
some neighbourhoods of $0$ and $\pi$,
$\mu_1(\alpha,\varphi,\psi)\neq \mu_2(\alpha,\varphi,\psi)$ and
\eqref{eq:expcasei}, together with Remark \ref{rem:periodicity},
yields the conclusion.

If alternative (ii) of Corollary \ref{cor:m=2-simm} holds, then
\eqref{eq:cara-C} and Lemma \ref{l:prop-G} imply that
\begin{align*}
&    \mathcal C(0,\varphi_1)=
                 2(\beta(\varphi_1))^2 G_{k_1}(0)<0,\quad
\mathcal C(\pi,\varphi_1)=
    2(\beta(\varphi_1))^2 G_{k_1}\left(\tfrac k2\pi\right)>0,\quad
                 \text{if $\omega(\varphi_1)=0$},\\
&    \mathcal C(0,\varphi_1)=
    2(\beta(\varphi_1))^2 G_{k_1}\left(-\tfrac \pi2\right)>0,\quad 
\mathcal C(\pi,\varphi_1)=
                 2(\beta(\varphi_1))^2 G_{k_1}\left(\tfrac k2 \pi-\tfrac\pi2\right)<0,\quad
                 \text{if $\omega(\varphi_1)=\tfrac{\pi}{k_1}$}.
\end{align*}
Then, for either $\alpha=0$ or $\alpha=\pi$ (and, by continuity, in
some neighbourhoods of $0$ and $\pi$),
$\mathcal C(\alpha,\varphi_1)\neq0$. Hence the conclusion follows from
\eqref{eq:neg}--\eqref{eq:pos} and Remark \ref{rem:periodicity},.
\end{proof}

\subsection{Symmetries of the rectangle}
Recalling the reflection through the $x_1$-axis $\sigma$ introduced in
\S \ref{sec:axially-symm-doma} and denoting as
$\sigma'(x_1,x_2)=(-x_1,x_2)$ the reflection through the $x_2$-axis,
we say that $\Omega$ has the symmetries of a rectangle if
\begin{equation}\label{eq:simm-rett}
  \Omega=\sigma(\Omega)=\sigma'(\Omega).
\end{equation}
In \cite[Proposition 5.3]{BNHHO2009} it is observed that, if $\Omega$
has the symmetries of a rectangle, then the first eigenvalue of
problem $(E_0)$ is double; furthermore, the multiplicity of any
eigenvalue of $(E_0)$ is even. We also observe that a domain $\Omega$
satisfying \eqref{eq:simm-rett} is invariant under rotations of a
$\pi$-angle, so that $\lambda^a_k=\lambda_k^{-a}$ for every
$a\in\Omega$ and $k\in\N\setminus\{0\}$. Moreover, $\Omega$ is axially
symmetric with respect to both the $x_1$-axis and the $x_2$-axis;
therefore Corollary \ref{c:symmdom} guarantees that any double
eigenvalue of $(E_0)$ bifurcates into two simple eigenvalues of
\eqref{eq:eige_equation_a}, as the pole $a$ approaches the origin
along one of the two axes.  This, combined with the analyticity of the
restrictions to the lines through the origin observed in Remark
\ref{rem:analiticita-multipli}, allows us to observe that the two
branches into which the eigenvalue bifurcates 
are on opposite sides with respect to the limit eigenvalue, as
stated in Corollary \ref{cor:simmetria-rettangolo} and proved below.
 \begin{proof}[Proof of Corollary \ref{cor:simmetria-rettangolo}]
   The hypothesis of rectangular symmetry implies, in particular, that
   $\Omega$ is symmetric with respect to the $x_1$-axis. Therefore,
   the conclusion of Corollary \ref{c:symmdom} holds for a certain odd
   number $k$ and an interval $I$ of positive length such that
   $0\in I$.  In particular, for some $i\in \{N,N+1\}$, we have, for
   every $\alpha\in I$,
\begin{equation}\label{eq:svirett}
    \lambda_i^{t(\cos\alpha,\sin\alpha)}=\lambda_N^0+C(\alpha)t^k+o(t^k)\quad\text{as }t\to0^+,
\end{equation}
for some $C(\alpha)\neq0$.  By Remark \ref{rem:analiticita-multipli},
for every $\alpha\in I$ there exists a function
$t\mapsto \Lambda_\alpha(t)$ such that
$\Lambda_\alpha(0)=\lambda^0_N$, $\Lambda_\alpha$ is analytic in a
neighborhood of $0$,
$\Lambda_\alpha(t)=\lambda_i^{t(\cos\alpha,\sin\alpha)}$ for $t>0$
sufficiently small, and $\Lambda_\alpha(t)$ is equal to either
$\lambda_N^{t(\cos\alpha,\sin\alpha)}$ or
$\lambda_{N+1}^{t(\cos\alpha,\sin\alpha)}$ for $t<0$ small.  From
\eqref{eq:svirett} and analyticity of $\Lambda_\alpha$ it follows that
$\Lambda_\alpha(t)=\lambda_N^0+C(\alpha)t^k+o(t^k)$ as $t\to0$, so
that, by oddness of $k$,
  \begin{equation}\label{eq:expL}
      \Lambda_\alpha(t)=\lambda_N^0-C(\alpha)|t|^k+o(|t|^k)\quad\text{as } t\to0^-.
  \end{equation}
  By Remark \ref{rem:analiticita-multipli} there exists
  $j\in\{N,N+1\}$ such that
  $\Lambda_\alpha(t)=\lambda_j^{t(\cos\alpha,\sin\alpha)}$ for $t<0$
  small.  Since $\Omega$ is invariant under rotations of a
  $\pi$-angle, we have
  $\lambda_j^{t(\cos\alpha,\sin\alpha)}=\lambda_j^{-t(\cos\alpha,\sin\alpha)}
  =\Lambda_\alpha(-t)$ for $t>0$ sufficiently small, so that by
  \eqref{eq:expL}
      \begin{equation}\label{eq:svirett2}
    \lambda_j^{t(\cos\alpha,\sin\alpha)}=\lambda_N^0-C(\alpha)t^k+o(t^k)\quad\text{as }t\to0^+.
 \end{equation}
 Since $C(\alpha)\neq0$, \eqref{eq:svirett} and \eqref{eq:svirett2}
 imply that $i\neq j$ (either $i=N$ and $j=N+1$ or $i=N+1$ and $j=N$),
 thus yielding the conclusion for $\alpha\in I$ with
 $\mu(\alpha)=|C(\alpha)|$. Finally, the symmetries of the problem
 allow us to conclude for all $\alpha \in I+\frac{\pi}2 \Z$.
  \end{proof}

\subsection{The case of the disk}

Arbitrarily increasing the number of axial symmetries or rotations one
can reach the case of the disk, which is clearly symmetric with
respect to any line passing through the origin, as well as any
rotation in the plane.  We denote
$D := \{(x_1,x_2) \in \R^2:x_1^2+x_2^2=1\}$.

As already pointed out in \cite[Section 3.2]{Abatangelo2019}, the
Aharonov--Bohm operator on the disk presents several peculiar
properties.

\begin{lemma}\label{l:diskeigenfunctions}
  If the pole is located at the origin, all the eigenvalues of $(E_0)$
  with $\Omega=D$ are double. The set of eigenvalues is precisely
  $\{z_{n,k}^2:n\in \N, k\in \N \text{ odd}\}$, where the
  $\{z_{n,k}\}_{n\in\N}$ are the zeros of the Bessel function
  $J_{k/2}$. For every $k\in\N$ odd and $n\in \N$, the eigenspace
  associated with the eigenvalue $\lambda=z_{n,k}^2$ is generated by
  the two $K$-real functions
    \begin{equation}\label{eq:uv}
        u(r\cos t,r\sin t)=  B e^{i \frac{t}{2}} J_{\frac{k}{2}}(\sqrt{\lambda}r) 
\sin \left(\tfrac{k}{2}t\right), \quad
v(r \cos t, r\sin t) = - B e^{i \frac{t}{2}} J_{\frac{k}{2}}(\sqrt{\lambda}r) 
\cos \left(\tfrac{k}{2}t\right) ,
\end{equation}
for $r\in(0,1)$ and $t\in[0,2\pi)$, where the constant $B>0$
(depending on $\lambda$ and $k$) is chosen in such a way that
$\|u\|_{L^2_K(\Omega)}=\|v\|_{L^2_K(\Omega)}=1$; moreover, the
functions $u$ and $v$ in \eqref{eq:uv} have the following asymptotic
behavior at $0$:
\begin{align}
  &u(r\cos t,r\sin t)= \beta e^{i \frac{t}{2}} r^{k/2} 
\sin (\tfrac{k}{2}t) + O(r^{\tfrac{k+4}2}), \label{eq:asyeigenfudisk}\\
&v(r \cos t, r\sin t) = \beta  e^{i \frac{t}{2}} r^{k/2} 
\sin (\tfrac{k}{2}(t-\tfrac{\pi}{k})) + O(r^{\tfrac{k+4}2}), \label{eq:asyeigenfvdisk}
\end{align}
 as  $r \to 0^+$,
for a certain $\beta\in\R$ (depending on $\lambda$ and $k$).
\end{lemma}
\begin{proof}
The first part of the lemma,
    including the characterization in \eqref{eq:uv} is proved in
    \cite[Lemma 2.4]{Abatangelo2019}. We observe that the constant $B$
    is the same for $u$ and $v$ by
    normalization in $L^2_K(\Omega)$. The expansions in
    \eqref{eq:asyeigenfudisk}--~\eqref{eq:asyeigenfvdisk} then follow
    directly from the well-known series expansion of the Bessel
    function $J_{k/2}$.
\end{proof}

From Corollary \ref{cor:simmetria-rettangolo} 
and rotational invariance of the disk, 
we finally deduce Corollary \ref{cor:disk}.

\begin{proof}[Proof of Corollary \ref{cor:disk}]
 By Lemma \ref{l:diskeigenfunctions} $\lambda$ is double; hence there
  exists $N$ such that
  \begin{equation*}
  \lambda_{N-1}^0<\lambda=\lambda_N^0=\lambda_{N+1}^0<\lambda_{N+2}^0.
\end{equation*}
Since the disk has, in particular, the symmetry of a rectangle, for
$\alpha=0$ fixed the conclusion follows from Corollary
\ref{cor:simmetria-rettangolo} with $\mathcal M=\mu(0)$.  In case of
any $\alpha$, it is then sufficient to exploit the rotation invariance
of the disk by any angle, which implies that $\lambda_k^a=\lambda_k^b$
for all $k$ if $|a|=|b|$.
\end{proof}

\bigskip\noindent {\bf Acknowledgments.}  The authors are members of
GNAMPA-INdAM.  L. Abatangelo is partially supported by the PRIN 2022
project 2022R537CS \emph{$NO^3$ - Nodal Optimization, NOnlinear
  elliptic equations, NOnlocal geometric problems, with a focus on
  regularity}, founded by the European Union - Next Generation EU.
V. Felli is partially supported by the PRIN 2022 project 20227HX33Z
\emph{Pattern formation in nonlinear phenomena} granted by the
European Union -- Next Generation EU.

\bibliographystyle{acm}
\bibliography{biblio}

\end{document}